\documentclass[sn-mathphys,Numbered]{sn-jnl}
\usepackage{graphicx}%
\usepackage{multirow}%
\usepackage{amsmath,amssymb,amsfonts}%
\usepackage{amsthm}%
\usepackage{mathrsfs}%
\usepackage[title]{appendix}%
\usepackage{xcolor}%
\usepackage{textcomp}%
\usepackage{manyfoot}%
\usepackage{booktabs}%
\usepackage{algorithm}%
\usepackage{algorithmicx}%
\usepackage{algpseudocode}%
\usepackage{listings}%

\usepackage[T1]{fontenc}
\usepackage{latexsym} 
\usepackage{amsfonts, amsthm, amsmath,amssymb}
\usepackage{times}
\usepackage{xcolor}
\usepackage{t-angles}
\definecolor{refkey}{rgb}{0,0,1}
\definecolor{labelkey}{rgb}{1,0,0}
\usepackage{graphicx}
\usepackage{float}
\usepackage{mdframed}
\usepackage{ulem}
\usepackage{hyperref}
\usepackage{soul}

\definecolor{darkblue}{rgb}{0.0, 0.0, 0.55}
\definecolor{darkcerulean}{rgb}{0.03, 0.27, 0.49}
\definecolor{darkpowderblue}{rgb}{0.0, 0.2, 0.6}
\definecolor{britishracinggreen}{rgb}{0.0, 0.26, 0.15}

\usepackage[all]{xy}
\xyoption{2cell}
\newtheorem{thm}{Theorem}[section]
\newtheorem{prop}[thm]{Proposition}
\newtheorem{lem}[thm]{Lemma}
\newtheorem{cor}[thm]{Corollary}
\theoremstyle{definition}
\newtheorem{defn}[thm]{Definition}
\newtheorem{rem}[thm]{Remark}
\numberwithin{equation}{section}

\def\sq{q^{\frac{1}{2}}}
\def\sqi{q^{-\frac{1}{2}}}

\def\<{{\langle}} 
\def\>{{\rangle}} 
\def\ra{{\triangleleft}} 
\def\la{{\triangleright}} 

\def\note#1{{}}

\def\cU{{\mathcal U}}

\def\beq{\begin{equation}} 
	\def\eeq{\end{equation}} 
\def\cop{{\Delta}}

\def\id{\mathrm{id}}

\def \hx{\widehat{x}}

\def\fR{{\mathfrak R}}

\def\id{\mathrm{id}} 
\def\cU{{\mathcal U}}
\def\fK{\mathbb{K}}
\def\fC{\mathbb{C}}
\def\fRr{\mathbb{R}}
\def\fN{\mathbb{N}}
\def\fZ{\mathbb{Z}}

\def\be{\boldsymbol{e}}
\def\ra{{\triangleleft}} 
\def\la{{\triangleright}} 
\def\cU{{\mathcal U}}
\renewcommand{\emph}[1]{{\it #1}}

\theoremstyle{thmstyleone}%
\theoremstyle{thmstyletwo}%

\theoremstyle{thmstylethree}%

\begin{document}

\title[Braided symmetries of $SU_{q,\phi}(2)$ and Podle\'s Spheres.]{Braided symmetries of $SU_{q,\phi}(2)$ and Podle\'s Spheres.}


\author[1,2]{\fnm{Rafa\l{}} \sur{Bistroń}}\email{rafal.bistron@doctoral.uj.edu.pl}

\author[2]{\fnm{Andrzej} \sur{Sitarz}}\email{andrzej.sitarz@uj.edu.pl}

\affil[1]{\orgdiv{Institute of Theoretical Physics}, \orgname{Jagiellonian University}, \orgaddress{\street{prof.\ Stanis\l awa \L ojasiewicza 11}, \city{Kraków}, \postcode{30-348}, \country{Poland}}}

\affil[2]{\orgdiv{Doctoral School of Exact and Natural Sciences}, \orgname{Jagiellonian University}, \orgaddress{\street{prof.\ Stanis\l awa \L ojasiewicza 11}, \city{Kraków}, \postcode{30-348}, \country{Poland}}}


\abstract{	We present an explicit form of braided symmetries of the quantum spheres, by introducing a braided quantum Hopf algebra $\cU_{q, \phi}$ and demonstrating that they are braided  Hopf modules over this braided Hopf algebra. To obtain this result, we systematically develop braided versions of structures like pairing between Hopf algebras, left module algebra as well as their compatibility with the $*$-structure.}



\maketitle

	\section{Introduction}
	As fundamental objects that generalize groups and natural symmetries of noncommutative algebras, Hopf algebras and quantum groups have been the subject of extensive study in recent decades \cite{qGroups1, Hopf1, qGroups2}. In the context of quantum deformations of Lie groups with a complex-valued deformation parameter, their further generalization as objects in a braided category \cite{brGeo} has been recently revisited, specifically for the $SU_{q, \phi}(2)$ \cite{su2,podles} and for entire classes of algebras \cite{arek} or as symmetries of certain algebras \cite{BJR21}.
	
	Naturally, one would wonder if the braided Hopf algebras are algebraic symmetries in terms of action and coaction. Our inspiration comes from Podle\'s quantum spheres \cite{podles0}, which are Hopf module algebras for one of the most significant and well-known Hopf algebras, $\cU_{q}(su(2))$, the quantum deformation of the universal enveloping algebra of the $su(2)$ Lie algebra. We anticipate that understanding the braided version of these symmetries will be essential to moving forward with the construction of spectral geometries for braided versions of $SU_q(2)$ and Podle\'s spheres. The reason for that is that the algebra $\cU_{q}(su(2))$ and its representation theory have played a significant role in the construction of spectral triples over both $SU_q(2)$ \cite{dirac} and the families of Podle\'s spheres \cite{diracds, dirac_all_podles}. While \cite{podles} provides the coaction-based solution for the symmetries of Podle\'s sphere the action-based solution has not been yet explored.
	
	A first natural approach is to construct a braided version of the quantum $\cU_{q}(su(2))$ that would give $SU_{q, \phi}(2)$ a structure of braided Hopf algebra modules over it.  These attempts have started over 20 years ago and few promising candidates have been proposed, see for example \cite{quasiT}. However, up to our knowledge, none of those possesses the  important geometrical properties of $\cU_{q}(su(2))$, 
	especially the pairing with  $SU_{q, \phi}(2)$ \cite{dirac} and the action on the algebra of quantum spheres.
	
	In this work, we explicitly design a braided Hopf algebra with a well-defined action on $SU_{q, \phi}(2)$ and a pairing with it. In order to do so, we extend the $\cU_{q}(su(2))$ algebra by permitting the generator $k$ to not be self-adjoint. This effectively makes the $u(2)$ Lie algebra resemble a braided version of the quantum deformation of a universal envelope.
	
	This paper begins with a discussion of the left action in the context of braided Hopf algebras, giving two key examples: the braided form of the left adjoint action and the action constructed from the braided pairing. The first is primarily a model of the braided action, whereas the latter is the result on which we base our subsequent work. 
	
	As a motivating example, we review the case of the $SU_{q, \phi}(2)$ braided Hopf algebra, presenting the braiding procedure  \cite{arek} which is used later. Then, we introduce the construction of the braided quantum Hopf algebra $\cU_{q,\phi}(\hat{u}(2))$ and provide the explicit braided pairing between $SU_{q, \phi}(2)$ and $\cU_{q,\phi}(\hat{u}(2))$ as well as the braided left action of $\cU_{q,\phi}(\hat{u}(2))$ on $SU_{q, \phi}(2)$.
	
	We study the irreducible representations of the algebra $\cU_{q,\phi}(\hat{u}(2))$ and check possible quadratic algebra structures over the irreducible representation of dimension 
	$3$. We arrive at two possible realisations, with only one of them admitting
	a star structure and bounded $\ast$-representations. The obtained family
	of algebras turns out to be, after appropriate basis change, the family
	of  Podle\'s spheres \cite{podles, podles0}. This gives an alternative
	derivation of Podle\'s spheres (compared to the one in \cite{podles}), arising 
	from imposing a Hopf module structure over $\cU_{q,\phi}(\hat{u}(2))$ in the braided sense.
	
	Finally, we investigate the $*$-structures for braided Hopf algebras and their modules, as well as the criteria for compatibility between left action and the $*$-structure for braided Hopf algebras derived through twisting.

\section{Braided Hopf module algebras}
In this section, we define the notion of a braided Hopf module algebra and show that a braided Hopf algebra 
acts on itself through an adjoint action. Then we present the construction of braided modules based on braided pairing between braided Hopf algebras.  We start by reviewing a graphic notation that we shall use to make 
calculations and proofs clearer and more articulate.

We assume that the reader is familiar with the notion of  a braided Hopf algebra, however, for completeness
we recall the definition.

\begin{defn}
We say  that $(H, \cdot , \cop, S, \epsilon, \eta, \Psi)$ is a braided Hopf algebra, if there exists a braiding  
\begin{equation*}
	\hstretch 90  \vstretch 60
	\begin{aligned}
		& \Psi \quad \rightarrow \quad
		\begin{tangle}
			\object{H} \step[2] \object{H}\\
			\x  {{{}^\Psi}} \step[1.21] \\
			\object{H} \step[2] \object{H} \\
		\end{tangle} \hspace{0.1 cm},
	\end{aligned}
\end{equation*}
such that the coproduct $\cop$ and counit $\epsilon$ are morphisms and the antipode $S$ is an antihomomorphism
in the braided category, satisfying the standard requirements of the Hopf algebra. By $\eta$ we denote 
the unit of the algebra. In the rest of the paper we use the standard notation for the operations, the product 
and the unit, 
\begin{equation*}
	\hstretch 90  \vstretch 60
	\begin{aligned}
		& \cdot \quad \rightarrow \quad
		\begin{tangle}
			\object{H} \step[2] \object{H}\\
			\cu \\
			\step[1] \object{H} 
		\end{tangle}\hspace{0.4 cm}, &&\text{~~~~~~} &&
		\eta \quad \rightarrow \quad 
		\begin{tangle}
			\unit \\
			\object{H}
		\end{tangle} \hspace{0.4 cm},\\
	\end{aligned}
\end{equation*}
	The coproduct, counit and the antipode, 
\begin{equation*}
	\hstretch 90  \vstretch 60
	\begin{aligned}
		& \Delta \quad \rightarrow \quad
		\begin{tangle}
			\step[1] \object{H}\\
			\cd \\
			\object{H} \step[2] \object{H} 
		\end{tangle} \hspace{0.4 cm}, && \text{~~~~~~}&&
		\epsilon \quad \rightarrow \quad 
		\begin{tangle}
			\object{H} \\
			\counit \\ 
		\end{tangle} \hspace{0.4 cm}, && \text{~~~~~~}&&
			& S \quad \rightarrow \quad
	\begin{tangle}
		\object{H}\\
		\S \\
		\object{H} 
	\end{tangle}
	\end{aligned}
\end{equation*}
If the algebra $H$ has a star structure $\ast$, we shall denote it as, 
\begin{equation*}
\hstretch 90  \vstretch 60
	\begin{aligned}
		* \quad \rightarrow \quad 
		\begin{tangle}
			\object{H} \\
			\O* \\ 
			\object{H}
		\end{tangle} \hspace{0.4 cm}.\\
	\end{aligned}
\end{equation*}
\end{defn}
More details on the properties of the braided antipode can be found in the Appendix \ref{branti}.
Next, we make the concept of the left braided action precise.  
\begin{defn}
\label{Ract}
Let $H$ be a braided Hopf algebra and $A$ an algebra. We say that $A$ is left braided-$H$ module algebra 
if there exists  a bilinear map, $\la : H \otimes A \to A$, called left action, and a braiding 
$\Xi: H \otimes A  \to A \otimes H$,
denoted graphically:
\begin{equation*}
	\la  \quad \rightarrow \quad 
	\hstretch 90  \vstretch 60
	\begin{tangle}
	\object{H} \step[2] \object{A} \\
	\lu[2] \\ 
	\step[2] \object{A}
	\end{tangle} \hspace{0.2 cm}, \hspace{1cm}
 \Xi \quad \rightarrow \quad   
		\begin{tangle}
			\object{H} \step[2] \object{A}\\
			\x  {{{}^\Xi}} \step[1.21] \\
			\object{A} \step[2] \object{H} \\
		\end{tangle} \hspace{0.1 cm}.  \\
\end{equation*}
such that the following compatibility conditions are satisfied for the braiding $\Xi$: 
\begin{equation}
	\label{xi-prod}
	\hstretch 90  \vstretch 60
	\begin{aligned}
	&\hbox{$\bullet$ compatibility with the product in $H$}: \hspace{5mm}
	& \begin{tangle}
			\object{H} \step[2] \object{H}\step[2]\object{A}\\
			\cu \step[2] \id \\
			\step[1] \nw1 \step[2] \id \\
			\step[2] \x {{{}^\Xi}} \\
			\step[2] \id  \step[2]\id \\
			\step[2] \object{A} \step[2] \object{H}\\
		\end{tangle} 
	 \quad = \quad 
	 \begin{tangle}
		\object{H} \step[2] \object{H}\step[2]\object{A}\\
		\id \step[2] \id \step[2] \id \\
		\id \step[2] \x {{{}^\Xi}} \\
		\x {{{}^\Xi}}  \step[1.31] \id \\
		\id \step[2] \cu \\
		\object{A} \step[3] \object{H} \\
	\end{tangle} \\[5mm]
	&\hbox{$\bullet$ compatibility with the product in $A$}: 
		&	\begin{tangle}
			\object{H} \step[2] \object{A}\step[2]\object{A}\\
			\id \step[2] \cu \\
			\nw1 \step[2] \id \\
			\step[1] \x {{{}^\Xi}} \\
			\step[1] \id  \step[2]\id \\
			\step[1] \object{A} \step[2] \object{H}\\
		\end{tangle}
		\quad = \quad 
		\begin{tangle}
			\object{H} \step[2] \object{A}\step[2]\object{A}\\
			\id \step[2] \id \step[2] \id \\
			\x {{{}^\Xi}}  \step[1.31]  \id \\
			\id \step[2] \x {{{}^\Xi~}}  \\
			\cu \step[2] \id  \\
			\step[1] \object{A} \step[3] \object{H} \\
		\end{tangle} \\[5mm]
&\hbox{$\bullet $ compatibility with the coproduct in $H$}: 
			&\begin{tangle}
				\step[1] \object{H}\step[3]\object{A}\\
				\cd \step[2] \id \\
				\id  \step[2] \x {{{}^\Xi}} \\
				\x {{{}^\Xi}} \step[1.31] \id \\
				\id \step[2] \id  \step[2]\id \\
				\object{A} \step[2] \object{H} \step[2] \object{H}\\
			\end{tangle}
			\quad = \quad
			\begin{tangle}
				\object{H}\step[2]\object{A}\\
				\id \step[2] \id \\
				\x {{{}^\Xi}}\\
				\id  \step[2] \nw1 \\
				\id  \step[2] \cd \\
				\object{A} \step[2] \object{H} \step[2] \object{H}\\
			\end{tangle}
			\end{aligned}
\end{equation}
Moreover, the left action $\la$ is compatible with the products in $A$ and $H$, 
\begin{equation}
			\label{action}
			\hstretch 90  \vstretch 60
			\begin{aligned}
				\begin{tangle}
					\object{H}\step[2]\object{A} \step[2]\object{A} \\
					\id  \step[2] \id  \step[2] \id \\
					\id  \step[2] \cu  \\
					\hlu[6] \\
					\step[3] \id \\
					\step[3] \id \\
					\step[3] \object{A}
				\end{tangle}
				\quad = \quad
				\begin{tangle}
					\step[1] \object{H}\step[3]\object{A} \step[2]\object{A} \\
					\cd \step[2] \id \step[2] \id \\
					\id \step[2] \x {{{}^\Xi}} \step[1.31] \id \\
					\id \step[2] \id \step[2] \id \step[2] \id \\
					\hlu[4] \step[2] \hlu[4] \\
					\step[2] \Cu \\
					\step[4] \object{A}
				\end{tangle}\hspace{0.2 cm}, \hspace{1 cm} &&
				&\begin{tangle}
					\step[1] \object{H}\step[2]\object{H}\step[2]\object{A}\\
					\step[1] \cu \step[2] \id  \\
					\step[2] \hlu[6]  \\
					\step[5] \id \\
					\step[5] \id \\
					\step[5] \object{A} 
				\end{tangle}
				\quad = \quad
				\begin{tangle}
					\step[1] \object{H}\step[2]\object{H}\step[2]\object{A}\\
					\step[1]\id \step[2] \id \step[2] \id   \\
					\step[1]\id \step[2] \hlu[4]   \\
					\step[1]\hlu[8]   \\
					\step[5] \id \\
					\step[5] \object{A} 
				\end{tangle} \hspace{0.2 cm}, &&
			\end{aligned}
		\end{equation}
and it is compatible with the braidings $\Psi, \Xi$,
		\begin{equation}
			\label{action_comp}
			\hstretch 90  \vstretch 60
			\begin{aligned}
				\begin{tangle}
					\object{H} \step[2] \object{H}\step[2]\object{A}\\
					\id \step[2] \id \step[2] \id \step[2]\\
					\id \step[2] \hlu[4] \\
					\nw2 \step[3] \id \\
					\step[2] \x {{{}^\Xi}} \\
					\step[2] \object{A} \step[2] \object{H}\\
				\end{tangle}
				\quad = \quad 
				\begin{tangle}
					\object{H} \step[2] \object{H}\step[2]\object{A}\\
					\id \step[2] \id \step[2] \id \step[2]\\
					\x {{{}^\Psi}} \step[1.21] \id \\
					\id \step[2] \x {{{}^\Xi}} \\
					\hlu[4] \step[2] \id \\
					\step[2] \object{A} \step[2] \object{H} \\
				\end{tangle}  \hspace{0.2 cm}.
			\end{aligned}
		\end{equation}
In case when $H$ and $A$ are unital  we require, of course, the usual properties concerning the action of the unit
 and on the unit:
\begin{equation}
\label{action_units}
1 \la a = a, \qquad \qquad h \la 1 = \epsilon(h) 1.	
\end{equation}
\end{defn}
\subsection{Example: the adjoint action}

A standard example of a left $H$-module algebra is given by $H$, with the adjoint action of $H$ on itself given by, 
	\begin{equation}
		\label{non_br_adj}
		h \la g = h_{(1)}g S(h_{(2)}), \quad \forall h,g \in H.
	\end{equation}	
This has again a natural generalization for braided Hopf algebras. 
\begin{prop}
Let $H$ be a braided Hopf algebra. The braided adjoint action of $H$ 
on itself is defined as
		\begin{equation}
			\label{br_adj}
			h \la g =   \cdot \circ (\cdot \otimes S)\circ(\id \otimes \Psi)\circ (\Delta \otimes \id)(h \otimes g) \vspace{0.4 cm}.
		\end{equation}
Together with the  $\Psi$ as the braiding between the braided Hopf algebra $H$ and the left $H$-module 
algebra $H$, it make $H$ a left braided $H$-module algebra in the sense of Definition \ref{Ract} .	
\end{prop}		
The action \eqref{br_adj} can be represented diagrammatically in an elegant way:
		\begin{equation}
			\label{br_adj_diag}
			\hstretch 90  \vstretch 60
			\begin{tangle}
				\object{H} \step[2] \object{H} \\
				\lu[2] \\
				\step[2] \object{H} \\
			\end{tangle}
			\quad = \quad
			\begin{tangle}
				\step[1] \object{H} \step[3] \object{H} \\
				\cd  \step[2] \id \\
				\id  \step[2]  \x {{{}^\Psi}} \step[1.21] \\
				\id \step[2]  \id \step[2] \S \\
				\cu \step[1] \ne1 \\
				\step[1] \cu \\
				\step[2] \object{H}  \step[1]  \\
			\end{tangle} \quad .
		\end{equation}
  
\begin{proof}
		Since the braiding  $\Xi$ is just the braiding $\Psi$ in $H$ it automatically satisfies
		the compatibility with the products and coproducts. Therefore, it remains to verify only the properties \eqref{action} and the condition \eqref{action_comp}.
		
		First, let us examine condition $f \la (g \la h) = (fg) \la h$ \eqref{action}  for the braided adjoin action \eqref{br_adj}. Starting from the right-hand side one obtains:
		
		\begin{equation*}
			\hstretch 90  \vstretch 60
			\begin{aligned}
				\begin{tangle}
					\object{H} \step[2] \object{H} \step[2] \object{H} \\
					\cu  \step[2] \id \\
					\step[1] \lu[3] \\
					\step[4] \object{H} \\
				\end{tangle}
				& \quad \underset{(1)}{=} \quad
				\begin{tangle}
					\object{H} \step[2] \object{H} \step[2] \object{H} \\
					\cu  \step[2] \id \\
					\cd  \step[2] \id \\
					\id  \step[2]  \x {{{}^\Psi}} \step[1.21] \\
					\id \step[2]  \id \step[2] \S \\
					\cu \step[1] \ne1 \\
					\step[1] \cu \\
					\step[2] \object{H} \\
				\end{tangle}
				\quad \underset{(2)}{=} \quad
				\begin{tangle}
					\step[1]\object{H} \step[4] \object{H} \step[3] \object{H} \\
					\cd  \step[2] \cd \step[2] \id \\
					\id  \step[2] \x {{{}^\Psi}} \step[1.21]  \id  \step[2] \id \\
					\cu \step[2] \cu \step[1] \ne1 \\
					\step[1]\id  \step[4]  \x {{{}^\Psi}} \step[1.21] \\
					\step[1]\id \step[3]  \ne2 \step[2] \S \\
					\step[1] \cu \step[3] \ne3 \\
					\step[2] \cu \\
					\step[3] \object{H} \\
				\end{tangle}
				\quad \underset{(3)}{=} \quad
				\begin{tangle}
					\step[1]\object{H} \step[4] \object{H} \step[3] \object{H} \\
					\cd  \step[2] \cd \step[2] \id \\
					\id  \step[2] \x {{{}^\Psi}} \step[1.21] \x {{{}^\Psi}} \step[1.21]   \\
					\cu \step[2] \x {{{}^\Psi}} \step[1.21] \id \\
					\step[1] \nw1  \step[2] \id \step[2]  \cu \\
					\step[2] \cu \step[3]  \S  \\
					\step[3] \id \step[3] \ne2 \\
					\step[3] \cu \\
					\step[4] \object{H} \\
				\end{tangle} \underset{(4)}{=} 
			\end{aligned}
		\end{equation*}
		\begin{equation*}
			\hstretch 90  \vstretch 60
			\begin{aligned}
				\phantom{xxx}& \underset{(4)}{=}
				\begin{tangle}
					\step[1]\object{H} \step[4] \object{H} \step[3] \object{H} \\
					\cd  \step[2] \cd \step[2] \id \\
					\id  \step[2] \x {{{}^\Psi}} \step[1.21] \x {{{}^\Psi}} \step[1.21]   \\
					\cu \step[2] \x {{{}^\Psi}} \step[1.21] \id \\
					\step[1] \nw1  \step[2] \id \step[2]  \x {{{}^\Psi}} \step[1.21] \\
					\step[2] \cu \step[2]  \S \step[2] \S \\
					\step[3] \id \step[3] \cu \\
					\step[3] \id \step[3] \ne2 \\
					\step[3] \cu \\
					\step[4] \object{H} \\
				\end{tangle}
				\quad \underset{(5)}{=} \quad
				\begin{tangle}
					\step[1] \object{H} \step[4] \object{H} \step[3] \object{H} \\
					\cd \step[2] \cd \step[2] \id \\
					\id \step[2] \id \step[2] \id  \step[2]  \x {{{}^\Psi}} \\
					\id \step[2] \x {{{}^\Psi}} \step[1.21]  \id \step[2] \S \\
					\cu \step[2] \x {{{}^\Psi}} \step[1.21] \id \\
					\step[1] \nw1 \step[2] \id \step[2] \x {{{}^\Psi}} \\
					\step[2] \cu \step[2] \id \step[2] \S \\
					\step[3] \id \step[3] \cu \\
					\step[3] \Cu \\
					\step[5] \object{H} \\
				\end{tangle}
				\quad \underset{(6)}{=} \quad
				\begin{tangle}
					\step[1] \object{H} \step[4] \object{H} \step[3] \object{H} \\
					\cd \step[2] \cd \step[2] \id \\
					\id \step[2] \id \step[2] \id  \step[2]  \x \\
					\id \step[2] \x {{{}^\Psi}} \step[1.21] \id \step[2] \S \\
					\id \step[2] \id \step[2] \x {{{}^\Psi}} \step[1.21] \id \\
					\id \step[2] \cu \step[2] \x {{{}^\Psi}} \\
					\id \step[3] \id \step[2] \ne1 \step[2] \S \\
					\nw2 \step[2] \cu \step[2] \ne2 \\
					\step[2] \cu \step[1] \ne1 \\
					\step[3] \cu \\
					\step[4] \object{H} \\
				\end{tangle}
				\quad  \\
		\end{aligned}
		\end{equation*}
		\begin{equation*}
	\hstretch 90  \vstretch 60
      \begin{aligned}
				& \underset{(7)}{=}  \quad
				\begin{tangle}
					\step[1] \object{H} \step[2] \object{H} \step[3] \object{H} \\
					\step[1] \id \step[1] \cd \step[2] \id \\
					\step[1] \id \step[1] \id  \step[2]  \x {{{}^\Psi}}  \\
					\step[1] \id \step[1] \id \step[2]  \id \step[2] \S \\
					\step[1] \id \step[1] \cu \step[1] \ne1 \\
					\step[1] \id \step[2] \cu \\
					\cd \step[2] \id \\
					\id  \step[2]  \x {{{}^\Psi}}  \\
					\id \step[2]  \id \step[2] \S \\
					\cu \step[1] \ne1 \\
					\step[1] \cu \\
					\step[2] \object{H} \\
				\end{tangle}
				\quad \underset{(8)}{=} \quad
				\begin{tangle}
					\object{H} \step[2] \object{H} \step[2] \object{H} \\
					\id  \step[2] \lu[2] \\
					\lu[4] \\
					\step[4] \object{H} \\
				\end{tangle} \quad .
			\end{aligned}
		\end{equation*}

		In the above in consecutive steps the following properties were used:
		
		\begin{enumerate}
			\setlength\itemsep{0 cm}
			\item Definition of braided adjoint action \eqref{br_adj_diag}.
			\item Morphism property of coproduct in a braided category.
			\item Compatibility of product with braiding.
			\item Antihomomorphism property of antipode in a braided category.
			\item Compatibility of antipode with braiding.
			\item Associativity of product.
			\item Compatibility of product with braiding.
			\item Definition of braided adjoint action \eqref{br_adj_diag}.
		\end{enumerate}
		
		The other property  from \eqref{action} is a braided analogue of the $f \la (gh) = (f_{(1)} \la g)(f_{(2)} \la h)$ with additional braiding between $g$ and $f_{(2)}$. 
		
		\begin{equation*}
			\hstretch 90  \vstretch 60
			\begin{aligned}
				\begin{tangle}
					\step[1] \object{H} \step[3] \object{H} \step[2] \object{H} \\
					\cd  \step[2] \id  \step[2] \id \\
					\id \step[2]  \x {{{}^\Psi}} \step[1.21] \id\\
					\lu[2] \step[2] \lu[2]\\
					\step[2] \Cu \\
					\step[4] \object{H} \\
				\end{tangle}
				& \quad \underset{(1)}{=} \quad
				\begin{tangle}
					\step[2] \object{H} \step[4] \object{H} \step[4] \object{H} \\
					\step[1]\cd  \step[3] \id  \step[4] \id \\
					\step[1] \id \step[2] \nw1 \step[2] \id \step[4] \id \\
					\step[1] \id \step[3]  \x {{{}^\Psi}} \step[3.21] \id\\
					\step[1] \id \step[3] \id \step[2] \nw1 \step[3] \id \\
					\cd \step[2] \id \step[2] \cd \step[2] \id \\
					\id  \step[2]  \x {{{}^\Psi}} \step[1.21] \id  \step[2]  \x {{{}^\Psi}} \\
					\id \step[2]  \id \step[2] \S \step[2] \id \step[2]  \id \step[2] \S \\
					\cu \step[1] \ne1 \step[2] \cu \step[1] \ne1\\
					\step[1] \cu \step[4] \cu  \\
					\step[2] \id \step[5] \ne2 \\
					\step[2] \Cu \\
					\step[4] \object{H} \\
				\end{tangle}
				\quad \underset{(2)}{=} \quad
				\begin{tangle}
					\step[3] \object{H} \step[5] \object{H} \step[2] \object{H} \\
					\step[1] \Cd \step[3] \id \step[2] \id \\
					\cd \step[2] \cd \step[2] \id \step[2] \id \\
					\id \step[2] \id \step[2] \id \step[2] \x {{{}^\Psi}} \step[1.21] \id \\
					\id \step[2] \id \step[2] \x {{{}^\Psi}} \step[1.21] \x {{{}^\Psi}} \\
					\id \step[2] \x {{{}^\Psi}} \step[1.21] \id \step[2] \id \step[2] \id \\
					\id \step[2] \id \step[2] \S \step[2] \id \step[2] \id \step[2] \S \\
					\cu \step[1] \ne1 \step[2] \cu \step[1] \ne1 \\
					\step[1] \cu \step[4] \cu \\  
					\step[2] \nw1 \step[4] \ne1 \\
					\step[3] \Cu \\
					\step[5] \object{H}
				\end{tangle}
				\quad \underset{(3)}{=}
			\end{aligned}
		\end{equation*}   
		\begin{equation*}
			\hstretch 90  \vstretch 60
			\begin{aligned}   
				\phantom{xxx} 
				& \underset{(3)}{=} \begin{tangle}
					\step[3] \object{H} \step[5] \object{H} \step[2] \object{H} \\
					\step[1] \Cd \step[3] \id \step[2] \id \\
					\ne1 \step[2] \Cd \step[1]\id \step[2] \id\\
					\id \step[2] \cd \step[3]  \hx {{{}^\Psi}}\step[1.21] \id \\
					\id \step[2] \id \step[2] \id \step[2] \ne1 \step[1] \id \step[2] \id \\
					\id \step[2] \id \step[2] \x {{{}^\Psi}} \step[1.21] \x {{{}^\Psi}} \\
					\id \step[2] \x {{{}^\Psi}} \step[1.21] \id \step[2] \id \step[2] \id \\
					\id \step[2] \id \step[2] \S \step[2] \id \step[2] \id \step[2] \S \\
					\cu \step[2] \cu \step[2] \cu \\
					\step[1] \Cu \step[3] \ne2 \\   
					\step[3] \Cu \\
					\step[5] \object{H}
				\end{tangle}
				\quad \underset{(4)}{=} \quad
				\begin{tangle}
					\step[2] \object{H} \step[4] \object{H} \step[2] \object{H} \\
					\Cd \step[2] \id \step[2] \id \\
					\id \step[3] \cd \step[1]\id \step[2] \id\\
					\id \step[2] \cd \step[1] \hx {{{}^\Psi}} \step[1.21] \id \\
					\id \step[2] \S \step[2] \id \step[1] \id \step[1] \id \step[2] \id  \\
					\id \step[2] \cu \step[1] \id \step[1] \x {{{}^\Psi}} \\
					\id \step[3] \x {{{}^\Psi}} \step[0.21] \id \step[2] \S \\
					\id \step[2] \ne1 \step[1] \ne1 \step[1] \id \step[2] \id \\
					\cu \step[2] \cu \step[1] \ne1 \\
					\step[1] \Cu \step[2] \id \\   
					\step[3] \Cu \\
					\step[5] \object{H}
				\end{tangle}
				\quad \underset{(5)}{=} \quad
				\begin{tangle}
					\step[2] \object{H} \step[4] \object{H} \step[2] \object{H} \\
					\Cd \step[2] \id \step[2] \id \\
					\id \step[3] \cd \step[1]\id \step[2] \id\\
					\id \step[3] \counit \step[2] \hx {{{}^\Psi}} \step[1.21] \id \\
					\id \step[3] \unit \step[2] \id \step[1] \x {{{}^\Psi}}  \\
					\id \step[3] \x {{{}^\Psi}} \step[0.21] \id \step[2] \S \\
					\id \step[2] \ne1 \step[1] \ne1 \step[1] \id \step[2] \id \\
					\cu \step[2] \cu \step[1] \ne1 \\
					\step[1] \Cu \step[2] \id \\   
					\step[3] \Cu \\
					\step[5] \object{H}
				\end{tangle} \underset{(6)}{=}
			\end{aligned}
		\end{equation*}   
		\begin{equation*}
			\hstretch 90  \vstretch 60
			\begin{aligned}   
				\phantom{xxx} &\underset{(6)}{=} \quad
				\begin{tangle}
					\step[1] \object{H} \step[3] \object{H} \step[2] \object{H} \\
					\cd \step[2] \id \step[2] \id \\
					\id \step[2] \x {{{}^\Psi}} \step[1.21] \id \\
					\id \step[2] \id \step[2] \x {{{}^\Psi}}  \\
					\nw1 \step[1] \cu \step[2] \S \\
					\step[1] \cu \step[3] \id \\
					\step[2] \Cu \\
					\step[4] \object{H}
				\end{tangle}
				\quad \underset{(7)}{=} \quad
				\begin{tangle}
					\step[1] \object{H} \step[2] \object{H} \step[2] \object{H} \\
					\cd \step[1] \cu \\
					\id \step[2] \x {{{}^\Psi}}  \\
					\id \step[2] \id \step[2] \S \\
					\cu \step[1] \ne1 \\
					\step[1] \cu \\
					\step[2] \object{H}
				\end{tangle}
				\quad \underset{(8)}{=} \quad 
				\begin{tangle}
					\object{H} \step[2] \object{H} \step[2] \object{H} \\
					\id \step[2] \cu \\
					\lu[3] \\
					\step[3] \object{H}
				\end{tangle} \, .
			\end{aligned}
		\end{equation*}

		Where in the consecutive steps the following properties were used:
		\begin{enumerate}
			\setlength\itemsep{0 cm}
			\item Definition of braided adjoint action \eqref{br_adj_diag}.
			\item Compatibility of coproduct with braiding.
			\item Associativity of product and coassociativity of coproduct.
			\item Compatibility of antipode and coproduct with braiding.
			\item Definition of antipode.
			\item Definition of unit and counit.
			\item Compatibility of product with braiding.
			\item Definition of braided adjoint action \eqref{br_adj_diag}.
		\end{enumerate}
		
		Then, as the last step we verify the property \eqref{action_comp},

		\begin{equation*}
			\hstretch 90  \vstretch 60
			\begin{aligned}
				\begin{tangle}
					\object{H} \step[2] \object{H}\step[2]\object{H}\\
					\id \step[2] \id \step[2] \id \step[2]\\
					\id \step[2] \hlu[4] \\
					\nw2 \step[3] \id \\
					\step[2] \x {{{}^\Psi}} \\
					\step[2] \object{H} \step[2] \object{H}\\
				\end{tangle}
				\quad \underset{(1)}{=} \quad 
				\begin{tangle}
					\object{H} \step[2] \object{H}\step[3]\object{H}\\
					\id \step[1] \cd  \step[2] \id \\
					\id \step[1] \id  \step[2]  \x {{{}^\Psi}}  \\
					\id \step[1] \id \step[2]  \id \step[2] \S \\
					\id \step[1] \cu \step[1] \ne1 \\
					\nw1 \step[1] \cu \\
					\step[1] \x {{{}^\Psi}} \\
					\step[1] \object{H} \step[2] \object{H}\\
				\end{tangle}
				\quad \underset{(2)}{=} \quad 
				\begin{tangle}
					\object{H} \step[3] \object{H}\step[3]\object{H}\\
					\id \step[2] \cd  \step[2] \id \\
					\x {{{}^\Psi}} \step[1.21] \x {{{}^\Psi}} \\
					\id \step[2] \x {{{}^\Psi}} \step[1.21] \S \\
					\id \step[2] \id \step[2] \x {{{}^\Psi}} \ \\
					\cu \step[1] \ne1  \step[2] \id \\
					\step[1] \cu \step[3] \id \\
					\step[2] \object{H} \step[4] \object{H}\\
				\end{tangle}
				\quad \underset{(3)}{=} \quad 
				\begin{tangle}
					\step[1] \object{H} \step[2] \object{H}\step[3]\object{H}\\
					\step[1] \x {{{}^\Psi}} \step[2.21]  \id \\
					\step[1] \id \step[2] \nw1 \step[2] \id \\
					\cd  \step[2] \x {{{}^\Psi}}  \\
					\id  \step[2]  \x {{{}^\Psi}} \step[1.21] \id \\
					\id \step[2]  \id \step[2] \S \step[2] \id \\
					\cu \step[1] \ne1 \step[2] \id \\
					\step[1] \cu  \step[3] \id \\
					\step[2] \object{H} \step[4] \object{H}\\
				\end{tangle}
				\quad \underset{(4)}{=} \quad 
				\begin{tangle}
					\object{H} \step[2] \object{H}\step[2]\object{H}\\
					\id \step[2] \id \step[2] \id \step[2]\\
					\x {{{}^\Psi}} \step[1.21] \id \\
					\id \step[2] \x {{{}^\Psi}} \\
					\hlu[4] \step[2] \id \\
					\step[2] \object{H} \step[2] \object{H} \\
				\end{tangle} \, .
			\end{aligned}
		\end{equation*}
		\phantom{xxx}
		In the above in consecutive steps the following properties were used:
		\begin{enumerate}
			\setlength\itemsep{0 cm}
			\item Definition of braided adjoint action \eqref{br_adj_diag}.
			\item  Compatibility of product with braiding.
            \item  Compatibility of antipode and coproduct with braiding.
			\item Definition of braided adjoint action \eqref{br_adj_diag}.
		\end{enumerate}
	\end{proof}
	\subsection{Example: an action from the pairing}
	
	Let us recall that the duality pairing between two Hopf algebras $H,A$,  is a linear map $H \otimes A \to \fK$, which satisfies 
	\begin{equation}
		\label{pairingcond}	
		\begin{aligned}
			& \langle hg, a\rangle = \langle  h , a_{(1)} \rangle  \langle g , a_{(2)} \rangle, \qquad
			&& \langle h, ab \rangle =  \langle  h_{(1)} , a \rangle  \langle h_{(2)} , b \rangle, \\
			& \langle h, S(a)\rangle = 	 \langle S(h), a\rangle, 
			&&  \langle(Sh)^*,a\rangle = \overline{\langle h,a^*\rangle}.
		\end{aligned}
	\end{equation}
	for all $h,g \in H$, $a,b \in A$.
	The last condition is required if both $A,H$ are Hopf $\ast$-algebras. 
	If such pairing exists it gives rise to a natural left $H$ module structure 
	on $A$, given by
	\begin{equation}
		\label{non_br_la}
		h \la a = a_{(1)}\langle h, a_{(2)} \rangle.
	\end{equation}
	To generalize this left module structure for the braided case we first need to define an appropriate braided pairing.
	\begin{defn}\label{defbrpa}
		Let $H$ and $A$ be two braided Hopf algebras.  Bilinear map  $\langle \cdot, \cdot \rangle: H \otimes A \rightarrow \fK$ is called a braided pairing if it satisfies the following properties: 
		\begin{equation}
			\label{BrPairing}
			\hstretch 90  \vstretch 60
			\begin{aligned}
				\begin{tangle}
					\object{H} \step[3] \object{A} \step[2] \object{A} \\
					\id \step[3] \cu \\
					\id \step[4] \id \\
					\coRo~ \\
				\end{tangle}
				\quad &= \quad
				\begin{tangle}
					\step[1] \object{H} \step[3] \object{A} \step[2] \object{A} \\
					\cd \step[2] \id \step[2] \id \\
					\id \step[2] \x {{{}^\Xi}} \step[1.31] \id \\
					\coro~ \step[2]\coro~ \\
				\end{tangle} \hspace{0.2 cm}, \\
				\begin{tangle}
					\object{H} \step[2] \object{H} \step[3] \object{A} \\
					\cu \step[3] \id \\
					\step[1] \id \step[4] \id \\
					\step[1]\coRo~ \\
				\end{tangle}
				\quad &= \quad
				\begin{tangle}
					\object{H} \step[2] \object{H} \step[3] \object{A} \\
					\id \step[2] \id \step[2] \cd \\
					\id \step[2] \x {{{}^\Xi}} \step[1.31] \id \\
					\coro~ \step[2]\coro~ \\
				\end{tangle}\hspace{0.2 cm}, \\
				\hspace{0.5 cm} \\
				\begin{tangle}
					\object{H} \step[2] \object{A} \\
					\S \step[2] \id \\
					\coro~ \\
				\end{tangle}
				\quad &= \quad
				\begin{tangle}
					\object{H} \step[2] \object{A} \\
					\id \step[2] \S \\
					\coro~ \\
				\end{tangle}\hspace{0.2 cm},
			\end{aligned} 
		\end{equation}
		where the box represents the paring and $\Xi$, as before, denotes the braiding map between $H$ and $A$, satisfying the conditions (\ref{xi-prod}), together with analogous conditions for product and coproduct in $A$.  
		
	\end{defn}
	\begin{thm}\label{thh_br_par}
		If $\langle \cdot, \cdot\rangle$ is a braided paring between Hopf algebras $H$, $A$ as defined in the Definition \ref{defbrpa}
		that additionally satisfies,
		\begin{equation}
			\label{BrPair2}
			\hstretch 90  \vstretch 60
			\begin{aligned}
				&\begin{tangle}
					\object{H} \step[2] \object{A} \step[2] \object{A} \\
					\id \step[2] \x {{{}^\Psi}}  \\ 
					\x {{{}^\Xi}} \step[1.31] \id \\
					\id \step[2] \coro~ \\
					\object{A} \\
				\end{tangle}
				\quad = \quad
				\begin{tangle}
					\object{H} \step[2] \object{A} \step[2] \object{A} \\
					\coro~ \step[2]  \id \\
					\step[4] \object{A} \\
				\end{tangle} \hspace{0.2 cm}, \hspace{1 cm}&&
				\begin{tangle}
					\object{H} \step[2] \object{H} \step[2] \object{A} \\
					\x {{{}^\Psi}} \step[1.31] \id \\
					\id \step[2] \x {{{}^\Xi}} \step[1.31] \\
					\coro~ \step[2] \id \\
					\step[4] \object{H} \\
				\end{tangle}
				\quad = \quad
				\begin{tangle}
					\object{H} \step[2] \object{H} \step[2] \object{A} \\
					\id \step[2] \coro~ \\
					\object{H} \\
				\end{tangle} \quad ,
			\end{aligned} 
		\end{equation}
		then the following defines a braided left action of $H$ on $A$:    
		\begin{equation}
			\label{constr}
			\hstretch 90  \vstretch 60
			\begin{aligned}
				&\begin{tangle}
					\object{H} \step[2] \object{A}\\
					\hlu[4] \\
					\step[2] \id \\
					\step[2] \object{A}
				\end{tangle}
				\quad = \quad 
				\begin{tangle}
					\object{H} \step[3] \object{A} \\
					\id \step[2] \cd \\
					\x {{{}^\Xi}} \step[1.31] \id \\
					\id \step[2] \coro~ \\
					\object{A} \\
				\end{tangle} \quad .
			\end{aligned}
		\end{equation}
	\end{thm} 
	\noindent\begin{proof} 
	    By the definition of braided paring \ref{defbrpa}, braiding $\Xi$ satisfies compatibility with product and coproduct both in $A$ and $H$ \eqref{xi-prod}. Therefore
	    to demonstrate that the formula \eqref{constr} defines the braided action in the sense of the Definition \ref{Ract}, we  just need to check the properties of left action \eqref{action}, \eqref{action_comp}.
		First, we examine  \eqref{action}:
		
		\begin{equation*}
			\hstretch 90  \vstretch 60
			\begin{aligned}
				\begin{tangle}
					\object{H} \step[2] \object{H} \step[2] \object{A}\\
					\cu \step[2] \id \\
					\step[1] \hlu[6] \\
					\step[4] \id \\
					\step[4] \object{A}
				\end{tangle}
				& \quad \underset{(1)}{=} \quad 
				\begin{tangle}
					\object{H} \step[2] \object{H} \step[2] \object{A} \\
					\cu \step[1] \cd \\
					\step[1] \x {{{}^\Xi}} \step[1.31] \id \\
					\step[1] \id \step[2] \coro~ \\
					\step[1] \object{A} \\
				\end{tangle}
				\quad \underset{(2)}{=} \quad 
				\begin{tangle}
					\object{H} \step[2] \object{H} \step[3] \object{A}\\
					\id \step[2] \id \step[2] \cd \\
					\id \step[2] \x {{{}^\Xi}} \step[1.31] \nw1 \\
					\x {{{}^\Xi}} \step[1.31] \id \step[3] \id \\
					\id \step[2] \cu \step[3] \id \\
					\id \step[3] \coRo~ \\
					\object{A} \\
				\end{tangle}
				\quad \underset{(3)}{=} \quad 
				\begin{tangle}
					\object{H} \step[2] \object{H} \step[3] \object{A}\\
					\id \step[2] \id \step[2] \cd \\
					\id \step[2] \x {{{}^\Xi}} \step[1.31] \nw1\\
					\id \step[2] \id \step[2] \id \step[2] \cd \\ 
					\x {{{}^\Xi}} \step[1.31] \x {{{}^\Xi}} \step[1.31] \id \\
					\id \step[2]  \coro~ \step[2] \coro~\\
					\object{A} \\
				\end{tangle}
				\\[7 pt]
				& \qquad \underset{(4)}{=} \quad 
				\begin{tangle}
					\object{H} \step[2] \object{H} \step[4] \object{A}\\
					\id \step[2] \id \step[3] \cd \\
					\id \step[2] \id \step[2] \cd \step[1] \nw1 \\
					\id \step[2] \x {{{}^\Xi}} \step[1.31] \id \step[2] \id \\
					\x {{{}^\Xi}} \step[1.31] \x {{{}^\Xi}} \step[1.31] \id \\
					\id \step[2] \coro~ \step[2] \coro~ \\
					\object{A} \\
				\end{tangle}
				\quad \underset{(5)}{=} \quad 
				\begin{tangle}
					\object{H} \step[3] \object{H} \step[3] \object{A}\\
					\id \step[3] \id \step[2] \cd \\
					\id \step[3] \x {{{}^\Xi}} \step[1.31] \id \\
					\id \step[2] \cd \step[1] \coro~ \\
					\x {{{}^\Xi}} \step[1.31] \id \\
					\id \step[2] \coro~ \\
					\object{A}\\
				\end{tangle}
				\quad \underset{(6)}{=} \quad 
				\begin{tangle}
					\object{H} \step[2] \object{H} \step[2] \object{A}\\
					\id \step[2] \hlu[4] \\
					\hlu[8] \\
					\step[4] \id \\
					\step[4] \object{A}
				\end{tangle} \hspace{0.2 cm}.
			\end{aligned}
		\end{equation*}
		
		In the above in consecutive steps the following properties were used:
		
		\begin{enumerate}
			\setlength\itemsep{0 cm}
			\item Construction of the left action from the pairing (\ref{constr}).
			\item Compatibility of braiding $\Xi$ with the product (\ref{xi-prod}).
			\item Second property of the pairing (\ref{BrPairing}).
			\item Coassociativity of the coproduct.
			\item Compatibility of braiding $\Xi$ with the coproduct.
			\item Construction of the left action (\ref{constr}).
		\end{enumerate}
		
		Similarly, we proceed with the second property of the left action
		\eqref{action}: 
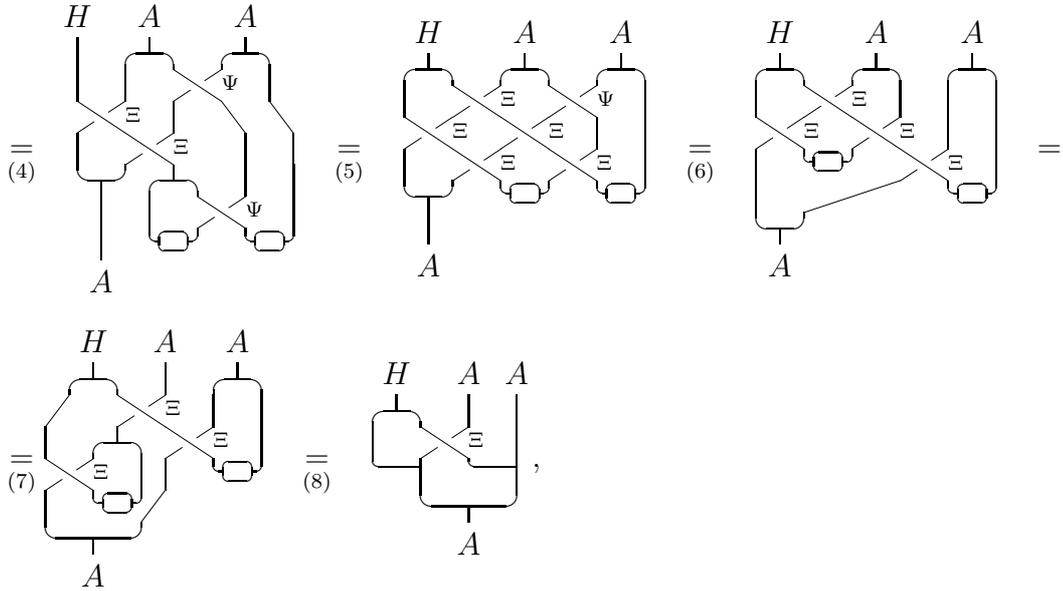
		\begin{equation*}
			\hstretch 90  \vstretch 60
			\begin{aligned}
				\begin{tangle}
					\object{H} \step[2] \object{A} \step[2] \object{A}\\
					\id \step[2] \cu \\
					\hlu[6] \\
					\step[3] \object{A}\\ 
				\end{tangle}
				&\quad \underset{(1)}{=} \quad 
				\begin{tangle}
					\object{H} \step[2] \object{A} \step[2] \object{A}\\
					\id \step[2] \cu \\   
					\id \step[2] \cd \\
					\x {{{}^\Xi}} \step[1.31] \id \\
					\id \step[2] \coro~ \\
					\object{A} \\
				\end{tangle}
				\quad \underset{(2)}{=} \quad 
				\begin{tangle}
					\object{H} \step[2] \object{A} \step[4] \object{A}\\
					\id \step[1] \cd  \step[2] \cd \\
					\id \step[1] \id \step[2] \x {{{}^\Psi}} \step[1.21] \id \\
					\id \step[1] \cu \step[2] \cu \\
					\x {{{}^\Xi}} \step[1.31] \step[2] \id \\
					\id \step[2] \coRo~ \\
					\object{A}
				\end{tangle}
				\quad \underset{(3)}{=} \quad 
				\begin{tangle}
					\object{H} \step[3] \object{A} \step[4] \object{A}\\
					\id \step[2] \cd  \step[2] \cd \\
					\id \step[2] \id \step[2] \x {{{}^\Psi}} \step[1.21] \id \\
					\x {{{}^\Xi}} \step[1.31] \id \step[2] \nw1 \step[1] \nw1\\ 
					\id \step[2] \x {{{}^\Xi}} \step[2.31]  \cu \\
					\cu \step[2] \coRo~ \\
					\step[1] \object{A} \\
				\end{tangle}
				\quad =
			\end{aligned}
		\end{equation*}
		\vfill
		\begin{equation*}
			\hstretch 90  \vstretch 60
			\begin{aligned}	
				\phantom{xssxx} & \quad  \underset{(4)}{=} \quad
				\begin{tangle}
					\object{H} \step[3] \object{A} \step[4] \object{A}\\
					\id \step[2] \cd  \step[2] \cd \\
					\id \step[2] \id \step[2] \x {{{}^\Psi}} \step[1.21] \id \\
					\x {{{}^\Xi}} \step[1.31] \id \step[2] \nw1 \step[1] \nw1\\ 
					\id \step[2] \x {{{}^\Xi}} \step[2.31]  \id \step[2] \id \\
					\cu \step[1] \cd \step[2] \id \step[2] \id \\
					\step[1] \id \step[2] \id \step[2]\x {{{}^\Xi}} \step[1.31] \id \\
					\step[1] \id \step[2] \coro~ \step[2] \coro~ \\
					\step[1] \object{A}\\
				\end{tangle}
				\quad \underset{(5)}{=} \quad 
				\begin{tangle}
					\step[1] \object{H} \step[4] \object{A} \step[4] \object{A}\\
					\cd \step[2] \cd  \step[2] \cd \\
					\id \step[2] \x {{{}^\Xi}} \step[1.31] \x {{{}^\Psi}} \step[1.21] \id \\
					\x {{{}^\Xi}} \step[1.31] \x {{{}^\Xi}} \step[1.31] \id \step[2] \id \\
					\id \step[2] \x {{{}^\Xi}} \step[1.31] \x {{{}^\Xi}} \step[1.31] \id \\
					\cu \step[2] \coro~ \step[2] \coro~ \\
					\step[1] \id \\
					\step[1] \object{A}
				\end{tangle}
				\quad \underset{(6)}{=} \quad 
				\begin{tangle}
					\step[1] \object{H} \step[4] \object{A} \step[4] \object{A} \\
					\cd \step[2] \cd  \step[2] \cd \\
					\id \step[2] \x {{{}^\Xi}} \step[1.31] \id \step[2] \id \step[2] \id\\
					\x {{{}^\Xi}} \step[1.31] \x {{{}^\Xi}} \step[1.31] \id \step[2] \id \\
					\id \step[2] \coro~ \step[2] \x {{{}^\Xi}} \step[1.31] \id \\
					\id \step[5] \ne4 \step[2] \coro~\\
					\cu \\
					\step[1] \object{A} \\
				\end{tangle}
				\quad = \\[10 pt]
				& \quad \underset{(7)}{=}
				\begin{tangle}
					\step[2] \object{H} \step[3] \object{A} \step[3] \object{A} \\
					\step[1] \cd \step[2] \id \step[2] \cd \\
					\ne1 \step[2] \x {{{}^\Xi}} \step[1.31] \id \step[2] \id \\
					\id \step[2] \cd \step[1] \x {{{}^\Xi}} \step[1.31] \id \\
					\x {{{}^\Xi}} \step[1.31] \id \step[1] \id \step[2] \coro~ \\
					\id \step[2] \coro~  \ne1 \\
					\Cu \\
					\step[2] \object{A} \\
				\end{tangle}
				\quad \underset{(8)}{=} \quad 
				\begin{tangle}
					\step[1] \object{H} \step[3] \object{A} \step[2] \object{A} \\
					\cd \step[2] \id \step[2] \id \\
					\id \step[2] \x {{{}^\Xi}} \step[1.31] \id \\
					\hlu[4] \step[2] \hlu[4] \\
					\step[2] \Cu \\
					\step[4] \object{A} \\
				\end{tangle} \hspace{0.2 cm},
			\end{aligned}
		\end{equation*}
		Where in the consecutive steps the following properties were used:
		\begin{enumerate}
			\setlength\itemsep{0 cm}
			\item Construction of the left action (\ref{constr}).
			\item  Morphism property of coproduct in a braided category. 
			\item Compatibility of braiding $\Xi$ with the product (\ref{xi-prod}).
			\item First property of the pairing (\ref{BrPairing}).
			\item Compatibility of braiding $\Xi$ with the coproduct (\ref{xi-prod}).
			\item First additional property of pairing (\ref{BrPair2}).
			\item Compatibility of braiding $\Xi$ with the coproduct (\ref{xi-prod}).
			\item Construction of the left action (\ref{constr}).
		\end{enumerate}
		
		\vfill
		Finally, we explicitly demonstrate the compatibility of the left action with braided map.
		
		\begin{equation*}
			\hstretch 90  \vstretch 60
			\begin{aligned}
				&\begin{tangle}
					\object{H} \step[2] \object{H}\step[2]\object{A}\\
					\id \step[2] \id \step[2] \id \step[2]\\
					\id \step[2] \hlu[4] \\
					\nw2 \step[3] \id \\
					\step[2] \x {{{}^\Xi}} \\
					\step[2] \object{A} \step[2] \object{H}\\
				\end{tangle}
				\quad \underset{(1)}{=} \quad 
				\begin{tangle}
					\object{H} \step[2] \object{H}\step[3]\object{A}\\
					\id \step[2] \id  \step[2] \cd \\
					\id \step[2] \x {{{}^\Xi}} \step[1.31] \id \\
					\x {{{}^\Xi}} \step[1.31] \coro~ \\
					\object{A} \step[2] \object{H}
				\end{tangle} 
				\quad \underset{(2)}{=} \quad 
				\begin{tangle}
					\step[1] \object{H} \step[2] \object{H}\step[2]\object{A}\\
					\step[1] \x {{{}^\Psi}} \step[1.21] \id \\
					\ne1 \step[2] \x {{{}^\Xi}} \step[1.31] \\
					\id \step[2] \cd \step[1] \nw1 \\
					\x {{{}^\Xi}} \step[1.31] \id \step[2] \id \\
					\id \step[2] \coro~ \step[2] \id \\
					\object{A} \step[6] \object{H} \\
				\end{tangle}
				\quad \underset{(3)}{=} \quad
				\begin{tangle}
					\object{H} \step[2] \object{H} \step[2] \object{A} \\  
					\x {{{}^\Psi}} \step[1.21] \id \\
					\id \step[2] \x {{{}^\Xi}} \\
					\hlu[4] \step[2] \id \\
					\step[2] \object{A} \step[3] \object{H}\\
				\end{tangle} \hspace{0.2 cm},
			\end{aligned}
		\end{equation*}
		where in (1) we used the construction of the left action (\ref{constr}), in (2) the second additional property of pairing (\ref{BrPair2}) and in (3)
		again the construction of the left action (\ref{constr}).
	\end{proof}
	\vfill

	\section{The algebra $\cU_{q,\phi}(\hat{u}(2))$}
	
	In this section, we discuss particular families of braided Hopf algebras, 
	which generalize the standard examples of $\cU_{q}(su(2))$ and $SU_{q}(2)$, preserving the relations between them. Because a large part of this section is based on \cite{arek}, we provide 
	a short summary of the most important results. 
	
	\begin{lem}\cite[Lemma 2.1]{arek}
		Let $H$ be a Hopf algebra and let $\chi: H \rightarrow \fC[z, z^{-1}]$ be a Hopf algebra homomorphism, then the following maps are equivariant under the coproduct in $H$:
		\begin{equation*}
			\begin{aligned}
				\Delta_L = (\chi \otimes \id)\Delta, ~~&&  \Delta_R = (\id \otimes \chi)\Delta .
			\end{aligned}
		\end{equation*}
	\end{lem}
	Having the maps $\Delta_L$, $\Delta_R$,  one can define an element $x \in H$ to be homogeneous of degrees $\mu(x)$, $\nu(x)$ iff 
	\begin{equation*}
		\begin{aligned}
			\Delta_L(x) = z^{\mu(x)}\otimes x,
			~~&& \Delta_R(x) = x \otimes z^{\nu(x)}.
		\end{aligned}
	\end{equation*}
	The function $\delta(x) := \mu(x) - \nu(x)$, defined for homogeneous $x$, is the called an index of $x$. Moreover, 
	any element $x$, which can be expressed as a product of homogeneous 
	elements $x = \prod_{i = 1}^n x_i$, is also homogeneous with 
	degrees $\mu(x)$, $\nu(x)$ and index $\delta(x)$ given by
	\begin{equation}
		\label{hom1}
		\begin{aligned}
			& \mu(x) = \sum_{i = 1}^n \mu(x_i), && \nu(x) = \sum_{i = 1}^n \nu(x_i), && \delta(x) = \sum_{i = 1}^n \delta(x_i).
		\end{aligned}
	\end{equation}

	\begin{thm} \cite[Theorems 2.13, 2.14]{arek}\label{Thm33} 
		Let $H$ be an Hopf algebra, $\chi: H \rightarrow \fC[z, z^{-1}]$ a Hopf algebra homomorphism   and $\phi \in [0, 2 \pi)$ an arbitrary phase. If $H$ has a countable basis of homogeneous elements, then it can be deformed into braided Hopf algebra with new actions defined in this homogeneous basis as follows
		\begin{equation*}
			\begin{aligned}
				& x \ast y  = e^{i \phi(\mu(x)\nu(y) - \mu(y)\nu(x))} xy,  &\qquad
				& \Delta_{\phi} (x) = \sum_{j} e^{ i \phi \delta(x_{(1)}^j) \delta(x_{(2)}^j) } x_{(1)}^j \otimes x_{(2)}^j, \\
				& S_{\phi} (x) = e^{ i \phi \delta(x)^2} S(x), \quad 
				&\epsilon_{\phi}(x) = \epsilon(x), \quad
				& \qquad \Psi (x \otimes y ) = e^{2 i \phi \delta(x) \delta(y)} y \otimes x .
			\end{aligned}
		\end{equation*}
	\end{thm}

	For more information about this method of constructing braided Hopf algebras, called twisting, we refer the reader to  \cite{arek}.
	
	\subsection{Quantum $SU_{q,\phi}(2)$}
	
	One of the simplest yet nontrivial examples of braided Hopf algebras is $SU_{q,\phi}(2)$. In \cite{arek} it was shown that such an algebra can be obtained by a  twisting from its well-known non-braided counterpart $SU_{q}(2)$ presented for example in \cite{dirac, su2} .
	
	Here we introduce $SU_{q,\phi}(2)$ using the notation similar to \cite{podles}. 
	
	\begin{defn}
		Let $0 \leq q \leq 1$ be a real number and $\phi \in [0, 2 \pi)$ an arbitrary phase. The braided Hopf algebra $SU_{q, \phi}(2)$ over $\fC$ is an algebra generated by $\alpha$, $\gamma$ and their adjoints $\alpha^*$, $\gamma^*$, subject to the following commutation relations:
		\begin{equation*}
			\begin{aligned}
				& \alpha \gamma = q e^{4 i \phi} \gamma \alpha, && \alpha \gamma^* = q e^{-4 i \phi} \gamma^* \alpha ,\\
				& \gamma \alpha^* = q e^{4i \phi} \alpha^* \gamma, && \gamma^* \alpha^* = q e^{-4 i \phi} \alpha^* \gamma^* ,\\
				& \alpha \alpha^* + q^2  \gamma^* \gamma = 1, & \qquad \alpha^* \alpha + \gamma \gamma^* = 1, \qquad & \gamma \gamma^* = \gamma^* \gamma .\\
			\end{aligned}   
		\end{equation*}
		The coproduct is defined on the generators as:
		\begin{equation*}
			\begin{aligned}
				& \Delta(\alpha) = \alpha \otimes \alpha - q e^{-4 i \phi} \gamma^* \otimes \gamma,  
				& \qquad &  \Delta(\gamma) = \gamma \otimes \alpha + \alpha^* \otimes \gamma, \\
				&\Delta(\alpha^*) = \alpha^* \otimes \alpha^* - q e^{-4 i \phi} \gamma \otimes \gamma^*,   
				&& \Delta(\gamma^*) = \gamma^* \otimes \alpha^* + \alpha \otimes \gamma^*, \\
			\end{aligned}
		\end{equation*}
		with the following counit and antipode,
		\begin{equation*}
			\begin{aligned}
				&\epsilon(\alpha) = 1, \qquad &&\epsilon(\alpha^*) = -1, \qquad &&\epsilon(\gamma) = 0,\qquad 
				&&\epsilon(\gamma^*) = 0, \\
				& S(\alpha) = \alpha^*, 
				&& S(\alpha^*) = \alpha ,
				&& S(\gamma) = -qe^{4 i \phi} \gamma ,
				&& S(\gamma^*) = -q^{-1} e^{4 i \phi}\gamma^* .\\
			\end{aligned}
		\end{equation*}
		and braiding:
		\begin{equation*}
			\begin{aligned}
				&\Psi(\gamma \otimes\gamma) = e^{8i\phi} \gamma \otimes\gamma, \qquad &&
				\Psi(\gamma^* \otimes \gamma^*)  = e^{8i\phi} \gamma^* \otimes\gamma^*, \\
				&\Psi(\gamma^* \otimes\gamma) = e^{-8i\phi} \gamma \otimes\gamma^*, \qquad &&
				\Psi(\gamma \otimes \gamma^*) = e^{-8i\phi} \gamma^* \otimes\gamma, \\
			\end{aligned}
		\end{equation*}
		which is trivial for all other pairs of generators and extends uniquely on the entire basis by property \eqref{hom1}. 
	\end{defn}
	
	\begin{rem}\cite[ Lemma 3.5]{arek}
		The braided Hopf algebra $SU_{q,\phi}$ is a deformation of $SU_q(2) = SU_{q,0}(2)$ following the scheme presented in Theorem \ref{Thm33} with
		\begin{equation*}
			\begin{aligned}	
				&\mu(\alpha) = 1, \quad &\mu(\alpha^*)=-1, \quad &\mu(\gamma) =-1, &&\mu(\gamma^*) = 1, \\
				&\nu(\alpha) = 1, \quad &\nu(\alpha^*)=-1, \quad &\nu(\gamma) =1, &&\nu(\gamma^*) = -1. 
			\end{aligned}
		\end{equation*}
		
	\end{rem}	
	
	\subsection{Quantum $\cU_{q,\phi}(\hat{u}(2))$}
	
	The classical counterpart of $SU_{q, \phi}(2)$, Hopf algebra $SU_{q}(2)$ has a non degenerate pairing with a Hopf algebra $\cU_{q}(su(2))$. However, to date, there have been no known braided analogues of $\cU_{q}(su(2))$ that preserves this paring. In this subsection we shall construct a braided analogue of $\cU_{q}(su(2))$, fulfilling the above properly, in the spirit of twisting.
	For the sake of completeness, we will start by invoking well-known definitions and properties that hold for (non-braided) $\cU_{q}(su(2))$.
	
	\begin{defn}
		Let $0 \leq q \leq 1$ be a real number. The Hopf algebra $\cU_q(su(2))$ over $\fC$ is an algebra generated by elements $e$, $f$, and an invertible element $k$, with product relations:
		
		\begin{equation}
			\label{uq_pr}
			\begin{aligned}
				& e k  = q k e, && kf = q fk, \\
				& e k^{-1}  = q^{-1} k^{-1} e, && k^{-1}f = q^{-1} fk^{-1}, \\
				& k^2 - k^{-2} = (fe - ef)(q - q^{-1}), \\
			\end{aligned}
		\end{equation}
		
		the coproduct defined by:
		
		\begin{equation}
			\label{uq_cpr}
			\begin{aligned}  
				&\Delta(k) = k \otimes k, && \Delta(k^{-1}) = k^{-1} \otimes k^{-1}, \\
				&\Delta(e) = e \otimes k + k^{-1} \otimes e, && \Delta(f) = f \otimes k  + k^{-1} \otimes f, \\ 
			\end{aligned}
		\end{equation}
		
		and counit $\epsilon$, antipode $S$ and $*$- structure:
		
		\begin{equation*}
			\begin{aligned}
				&\epsilon(k) = 1, &~~& S(k) = k^{-1}, &~~& k^* = k, \\
				&\epsilon(k^{-1}) = 1, &~~& S(k^{-1}) = k, &~~& (k^{-1})^* = k^{-1}, \\
				&\epsilon(f) = 0, &~~& S(f) =  -q f, &~~& f^* = e, \\
				&\epsilon(e) = 0, &~~& S(e) = -q^{-1} e, &~~& e^* = f. \\
			\end{aligned}
		\end{equation*}
	\end{defn}
	We recall the basic Theorem about the relation between 
	$\cU_q(su(2))$ and $SU_q(2)$ .
	
	\begin{thm} \cite{dirac}
		There exists a bilinear pairing between $\cU_q(su(2))$ and $SU_q(2)$ defined on generators by:
		\begin{equation*}
			\begin{aligned}
				& \langle e, \gamma^*\rangle = -q^{-1}, && \langle f,\gamma\rangle = 1, \\
				&\langle k, \alpha\rangle = q^{-1/2}, && \langle k, \alpha^*\rangle = q^{1/2}, \\
				&\langle k^{-1}, \alpha\rangle = q^{1/2}, && \langle k^{-1}, \alpha^*\rangle = q^{-1/2}, \\
			\end{aligned}
		\end{equation*}
		with all the other pairs of generators pairing to 0. 
		The paring gives rise to an action $\cU_q(su(2))$ on $SU_q(2)$ defined by $h \la a = a_{(1)} \langle h, a_{(2)} \rangle$, which is given explicitly on the generators as:
		\begin{equation*}
			\begin{aligned}
				& k \la \alpha = q^{-1/2} \alpha, \hspace{0.5 cm}&& k \la \alpha^* = q^{1/2} \alpha^*, \hspace{0.5 cm}&& k \la \gamma = q^{-1/2} \gamma, \hspace{0.5 cm}&& k \la \gamma^* = q^{1/2} \gamma^*,\\
				& k^{-1} \la \alpha = q^{1/2} \alpha, \hspace{0.5 cm}&& k^{-1} \la \alpha^* = q^{-1/2} \alpha^*, \hspace{0.5 cm}&& k^{-1} \la \gamma = q^{1/2} \gamma, \hspace{0.5 cm}&& k^{-1} \la \gamma^* = q^{-1/2} \gamma^*,\\
				& f \la \alpha = -q \gamma^*, \hspace{0.5 cm}&& f \la \alpha^* = 0, \hspace{0.5 cm}&& f \la \gamma = \alpha^*, \hspace{0.5 cm}&& f \la \gamma^* =0,\\
				& e \la \alpha = 0, \hspace{0.5 cm}&& f \la \alpha^* = \gamma, \hspace{0.5 cm}&& e \la \gamma = 0, \hspace{0.5 cm}&& e \la \gamma^* = -q^{-1} \alpha.\\
			\end{aligned}
		\end{equation*}
	\end{thm}
	
	To proceed with a suitable candidate for the braided version
	of $\cU_{q}(su(2))$ we need to acknowledge that this Hopf
	algebra cannot be braided using the same procedure, as there exists no nontrivial homomorphism $\chi: \cU_{q}(su(2)) \rightarrow \fC[z, z^{-1}]$, that would make the generators $e$, $f$ and $k$ homogeneous. Therefore we need to modify the structure of the algebra or, in fact, work with a slightly bigger algebra. To see how
	we derive the desired form let us assume that all maps in $\cU_q(su(2))$ differ from the standard ones only
	by a scalar factor of modulus one, with $\phi$ being a free 
	parameter. Moreover, we assume that in the case of $\phi = 0$ new braided Hopf algebra reduces to $\cU_q(su(2))$.
	
	Let us have a closer look at the constraints for the proposed braiding procedure.
	
	\begin{lem}\label{lem_br}
   Under the assumptions described above 
		the only non-zero braiding phases between generators $e$, $f$, $k$, $k^{-1}$ may be $\varphi(e,e)$, $\varphi(e,f)$, $\varphi(f,e)$ and $\varphi(f,f)$, with the following relation:
		\begin{equation}
			\varphi(e,e) = -\varphi(e,f) = -\varphi(f,e) = \varphi(f,f),
		\end{equation}
	\end{lem}
	\begin{proof}
		Assume that the braided coproduct of $k$ is equal to $\Delta_{\phi}(k) = c \Delta(k) = c k \otimes k$ where $c \ne 0$ is some unknown scalar. By compatibility of the braiding with the coproduct \eqref{xi-prod} for $k$ and any homogeneous $h \in \cU_{q}(su(2))$
		\begin{equation*}
			\begin{aligned}
				&\left( (\id \otimes \Delta_{\phi})\circ \Psi \right)(k \otimes h) = c e^{ i \phi \varphi(k,h)} h \otimes k \otimes k, \\
				&\left( \Psi \otimes \id) \circ(\id \otimes \Psi)\circ(\Delta_{\phi} \otimes \id)\right)(k \otimes h) = c e^{2 i \phi \varphi(k,h)} h \otimes k \otimes k,
			\end{aligned}
		\end{equation*}
		where $\circ$ stands for the composition of maps. 
		So if phase $\phi$ is arbitrary, then  $\varphi(k,h) = 0$ for all $h \in \cU_{q,\phi}(\hat{u}(2))$.
		The last relation from (\ref{uq_pr}), by compatibility of braiding with the product gives:
		\begin{equation*}
			\begin{aligned}
				&\Psi( h \otimes (fe - ef) ) = e^{i \phi (\varphi(h,f) + \varphi(h, e))} (fe - ef) \otimes h, \\
				&\Psi\left( h \otimes (k^2 - k^{-2}) \right) =   (k^2 - k^{-2}) \otimes h,
			\end{aligned}
		\end{equation*}
		and  analogously
		\begin{equation*}
			\begin{aligned}
				&\Psi((fe - ef) \otimes h ) = e^{i \phi (\varphi(f,h) + \varphi(e, h))} h \otimes (fe - ef), \\
				&\Psi\left( (k^2 - k^{-2}) \otimes h \right) =  h \otimes  (k^2 - k^{-2}),
			\end{aligned}
		\end{equation*}
		for any $h \in \cU_{q, \phi}(\hat{u}(2))$. To prove the claim of the lemma it's enough to substitute  $h$ by $e$ and by $f$.
	\end{proof}
	
	Based on those results, we tried to create a braided version of bialgebra structure from $\cU_q(su(2))$ with the help of language FORM$^{\text{\textregistered}}$ by checking all possibilities for braiding factors. However, those attempts were unsuccessful due to a bigger number of conditions when compared to the number of free parameters. The only consistent solution exists for $\phi = 0$ or $\phi = \pi$ which is not satisfying.  However, to solve this problem,
	we propose to extend the algebra by requiring that the adjoint of $k$,
	$k^*$ is an independent generator. The generator $k^*$ was added with the condition that relations with $k^*$ differ from those with $k$ only by a scalar factor of modulus $1$. Such construction allows not only the bialgebra, but the entire braided Hopf algebra structure 
	to be introduced which was again, effectively checked with the help of FORM$^{\text{\textregistered}}$.\footnote{Our calculations can be found at \url{https://github.com/RafalBistron/balchelor/blob/master/uqsu2.frm}}
	
	\begin{defn}
		Let $0 \leq q \leq 1$ be a real number and $\phi$ arbitrary phase. We define $\cU_{q,\phi}(\hat{u}(2))$ over $\fC$ as a $\ast$-algebra generated by $e$, $f$ and invertible elements $k$, $k^*$, with the following product relations (we omit the relations obtained by conjugation),
		\begin{equation}
			\label{eq5}
			\begin{aligned}
				&e k = q e^{-4i\phi} ke, \qquad  && k f = q e^{-4i\phi} fk,  \\
				& (e f -  f e) = \frac{1}{q^{-1}-q} \left( k k^* - k^{-1} (k^*)^{-1} \right), \qquad 
				&& kk^* = k^*k.
			\end{aligned}
		\end{equation}
		Then, with the following coproduct, 
		\begin{equation}
			\label{eq6}
			\begin{aligned}
				& \cop(k) = k \otimes k, \qquad  && \cop(k^*) = k^* \otimes k^*,  \\
				& \cop(e) = e \otimes k + k^{-1} \otimes e, \qquad 
				&& \cop(f) =  f \otimes k^* + (k^{*})^{-1} \otimes f, 
			\end{aligned}
		\end{equation}
		and the braiding,
		\begin{equation}
			\label{eq7}
			\begin{aligned}
				&\Psi(e \otimes f) = e^{-8i\phi} f \otimes e, \qquad &&
				\Psi(f \otimes e)  e^{-8i\phi} e \otimes f, \\
				&\Psi(e \otimes e) = e^{8i\phi} e \otimes e, \qquad &&
				\Psi(f \otimes f)  = e^{8i\phi} f \otimes f . \\
			\end{aligned}
		\end{equation}
        which vanish on other pairs of generators and extends uniquely on the entire basis by compatibility with the product,
		as well as the counit and the antipode:
		\begin{equation}
			\begin{aligned}
				& \epsilon(e) = 0, && \epsilon(f) = 0, && \epsilon(k) = 1 && \epsilon(k^*) = 1 \\
				&S(e) = -q^{-1} e^{4i\phi} e, && S(f) = -q e^{4i\phi} f,
				&&S(k) = k^{-1}, \qquad && S(k^*) = (k^*)^{-1}, \\
			\end{aligned}
		\end{equation}
		it becomes a  braided $*$- Hopf algebra.
	\end{defn}
	
	\begin{rem}\label{rem1}
		Observe that for $\phi=0$  the algebra $\cU_{q, 0}(\hat{u}(2))$ is 
		an ordinary $\ast$-Hopf algebra and there exists a Hopf algebra homomorphism 
		$\chi: \cU_{q, 0}(\hat{u}(2)) \rightarrow \fC[z, z^{-1}]$,
		\begin{equation*}
			\begin{aligned}
				& \chi(e) = 0, \qquad \qquad  && \chi(f) = 0, \\
				& \chi(k) = z,  && \chi(k^*) = z^{-1}, \\
			\end{aligned}
		\end{equation*}
		which leads to following gradings $\mu,\nu$ on   $\cU_{q, 0}(\hat{u}(2))$:
		
		\begin{equation*}
			\begin{aligned}
				& \mu(e) = -1, \quad  &\mu(f) = 1, \quad  \quad & \mu(k) = 1, \quad & \mu(k^*) = -1, \\
				& \nu(e) = 1, \quad  &\nu(f) = -1, \quad  \quad & \nu(k) = 1, \quad & \nu(k^*) = -1, \\
				& \delta(e) = -2, \quad  &\delta(f) = 2, \quad  \quad & \delta(k) = 0, \quad & \delta(k^*) = 0. 
			\end{aligned}
		\end{equation*}
		So the algebra, $\cU_{q,0}(\hat{u}(2))$ can be deformed by twisting, similarly as $SU_q(2)$, and the resulting braided 
		Hopf $\ast$-algebra is $\cU_{q,\phi}(\hat{u}(2))$.
	\end{rem}
	
	Note that if one perform polar decomposition $k = |k|U$, the bialgebra structure of $\cU_{q,\phi}(\hat{u}(2))$ can be embedded into greater bialgebra $\cU_{q,\phi}(u(2))$.
	
	\begin{defn}
		Let $\cU_{q,\phi}(u(2))$ be braided Hopf algebra over $\fC$ generated by $e$, $f = e^*$, invertible $|k| = |k|^*$, and $U$ such that $U^* = U^{-1}$, subject to the following relations:
		\begin{equation}
			\begin{aligned}
				&e |k| = q |k|e, \qquad  && |k| f = q  f|k|,  \\
				&e U = e^{-4 i \phi} Ue, \qquad  && U f = e^{-4 i \phi}  f U,  \\
				& (e f -  f e) = \frac{1}{q^{-1}-q} \left( |k|^2 - |k|^{-2} \right), \qquad 
				&& |k| U = U |k|.
			\end{aligned}
		\end{equation}
		With the coproduct defined by 
		\begin{equation}
			\begin{aligned}
				& \cop(|k|) = |k| \otimes |k|, \qquad  && \cop(U) = U \otimes U,  \\
				& \cop(e) = e \otimes |k|U + |k|^{-1}U^* \otimes e, \qquad 
				&& \cop(f) =  f \otimes |k|U^* + |k|^{-1}U \otimes f, 
			\end{aligned}
		\end{equation}
		the braiding (assumed to vanish on other pairs of generators and to extend uniquely on the entire basis by compatibility with thproduct),
		\begin{equation}
			\begin{aligned}
				&\Psi(e \otimes f) = e^{-8i\phi} f \otimes e, \qquad &&
				\Psi(f \otimes e) = e^{-8i\phi} e \otimes f, \\
				&\Psi(e \otimes e) = e^{8i\phi} e \otimes e, \qquad &&
				\Psi(f \otimes f)  = e^{8i\phi} f \otimes f ,
			\end{aligned}
		\end{equation}
		the counit and the antipode:
		\begin{equation}
			\begin{aligned}
				&S(e) = -q^{-1} e^{4i\phi} e, && S(f) = -q e^{4i\phi} f, && S(|k|) = |k|^{-1},  && S(U) = U^{-1}. \\
				& \epsilon(e) = 0, && \epsilon(f) = 0, && \epsilon(|k|) = 1, && \epsilon(U) = 1. \\
			\end{aligned}
		\end{equation}
	\end{defn}
	
	Unlike in the case of $\cU_q(u(2)) = \cU_q(su(2)) \oplus U(1)$, the braided Hopf algebra $\cU_{q, \phi}(u(2))$ cannot be simply decomposed into parts corresponding to modifications of $su(2)$ enveloping algebras and $U(1)$. Even when $\phi = 0$, decomposition into $\cU_q(su(2)) \oplus U(1)$ is true only for the algebra structure.  The algebra  $\cU_{q,\phi}(\hat{u}(2))$ is in fact 
	a natural sub-bialgebra of $\cU_{q, \phi}(u(2))$. 
	
	However, when $\phi = 0$, Hopf algebra $\cU_q(su(2))$ can be extracted from both $\cU_{q, 0}(\hat{u}(2))$ and $\cU_{q, 0}(u(2))$. In the first case,
	it is the quotient by a Hopf ideal generated by $k-k^*$ whereas in the second case, it is a quotient by a star Hopf ideal generated by $U-1$. The latter case could be described as an exact sequence of Hopf algebras:
	\begin{equation*}
	\mathbb{C}[U,U^*] \rightarrow \cU_{q, 0}(u(2)) 
	\rightarrow \cU_q(su(2)). 
	\end{equation*}
	Details on the $\phi \to 0$ and $q \to 1$ limits for $\cU_{q,\phi}(\hat{u}(2))$ may be found the Appendix \ref{claslim}.
	
	Moreover, since we have demonstrated that there is no unambiguous way 
	to find a braiding counterpart of $\cU_q(su(2))$ with a scalar braiding, we interpret $\cU_{q,\phi}(\hat{u}(2))$ as such analogue and in the next section we find the exact pairing and the action on $SU_{q,\phi}(2)$. 
	
	We would like to end this subsection by mentioning that the above braiding method, which "complexifies" selfadjoint generators corresponding to 
	the Cartan subalgebra, can be generalized for any Hopf algebra that is a quantum group originating from Lie algebra. 
	
	\subsection{Relations between $\cU_{q,\phi}(\hat{u}(2))$ and $SU_{q, \phi}(2)$}
	
	Next, we shall focus on the relationship between $\cU_{q,\phi}(\hat{u}(2))$ and $SU_{q, \phi}(2)$ braided Hopf algebras.
	
	\begin{lem}
		The following map $\Xi$ gives the braiding between the two braided Hopf algebras that is compatible with the products and coproducts both in $SU_{q,\phi}(2)$ and $\cU_{q,\phi}(\hat{u}(2))$:
		\begin{equation}
			\label{eq11}
			\Xi(h \otimes a) = e^{2 i \phi \delta(h)\delta(a)} a \otimes h,
		\end{equation}
		where $h \in \cU_{q,\phi}(\hat{u}(2))$, $a \in SU_{q,\phi}(2)$ and $\delta$ are as defined before.  Explicit form of braiding on the pairs of generators is given by 
		\begin{equation}
			\label{eq12}
			\begin{aligned}
				&\Xi(e \otimes \gamma) = e^{8 i \phi} \gamma \otimes e, \qquad&& \Xi(e \otimes \gamma^*) = e^{- 8 i \phi} \gamma^* \otimes e, \\
				&\Xi(f \otimes \gamma) = e^{-8 i \phi} \gamma \otimes f, \qquad  && \Xi(f \otimes \gamma^*) = e^{8 i \phi} \gamma^* \otimes f, 
			\end{aligned}
		\end{equation}
	where we've written only the pairs with nontrivial (different than transposition) braiding.
	\end{lem}
	
	The compatibility with the coproduct and product \eqref{xi-prod} in $\cU_{q,\phi}(\hat{u}(2))$  follows directly  from the properties of $\delta$. Similarly, it can be
	explicitly verified that the compatibility with the product and coproduct \eqref{xi-prod} in $SU_{q,\phi}(2)$ holds. We omit the straightforward calculations.
	
	\begin{thm}
		The following braided pairing between $\cU_{q,\phi}(\hat{u}(2))$ and $SU_{q,\phi}(2)$, which is defined below on the set of generators (with all other parings vanishing)
		satisfies all properties stated in  Theorem \ref{thh_br_par}, Eqs. \eqref{BrPair2}.
		\begin{equation}
			\label{eq8}
			\begin{aligned}
				& \langle e, \gamma^* \rangle  = -q^{-1} e^{- 4 i \phi}, \qquad && \langle f, \gamma\rangle  = 1, \\
				& \langle k,\alpha\rangle  = q^{-1/2} e ^{2 i \phi}, && \langle k,\alpha^*\rangle  = q^{1/2} e ^{-2 i \phi}, \\
				& \langle k^*,\alpha\rangle  = q^{-1/2} e ^{-2 i \phi}, && \langle k^*,\alpha^*\rangle  = q^{1/2} e ^{2 i \phi}, \\
				& \langle k^{-1},\alpha\rangle  = q^{1/2} e ^{-2 i \phi},&& \langle k^{-1},\alpha^*\rangle  = q^{-1/2} e ^{2 i \phi}, \\
				& \langle k^{*-1},\alpha\rangle  = q^{1/2} e ^{2 i \phi}, && \langle k^{*-1},\alpha^*\rangle  = q^{-1/2} e ^{-2 i \phi}. \\
			\end{aligned}
		\end{equation}
	\end{thm}
	We again omit the proof as it is a straightforward but tedious computation 
	and is based on explicit verifications of the pairings with the products 
	and the coproducts as required in (\ref{BrPairing}). 
	
	Using the above defined pairing by the Theorem \ref{thh_br_par} one can construct braided left action of
	$\cU_{q,\phi}(\hat{u}(2))$ on $SU_{q,\phi}(2)$. 
	
	\begin{thm}
		There exists a braided left action of the braided Hopf algebra $\cU_{q,\phi}(\hat{u}(2))$ on $SU_{q,\phi}(2)$,
		with the mutual braiding $\Xi$ \eqref{eq12}, defined through:
		\begin{equation}
			\label{eq13}
			\begin{aligned}
				h \la x = (\id \otimes <. , .>)(\Xi \otimes \id)(\id \otimes \Delta)(h \otimes x)~.
			\end{aligned}
		\end{equation}
		Furthermore there exists a right braided action, which is obtained with the inverse braiding i.e.
		\begin{equation}
			\label{eq16}
			\Xi^{-1}(a \otimes h) = e^{2 i \phi \delta(h)\delta(a)} h \otimes a
		\end{equation}
		and pairing redefined as a function $ SU_{q,\phi}(2) \otimes \cU_{q,\phi}(\hat{u}(2)) \rightarrow \mathbb{C}$. 
		Then, similarly, we have
		\begin{equation}
			\label{eq16_full}
			x \ra h = (<. , .> \otimes \id)(\id \otimes \Xi^{-1})(\Delta \otimes \id)(x \otimes h)~.
		\end{equation}
	\end{thm}
	The explicit formulae for the action on the generators are given in the Appendix 
	\ref{appact}.  To construct a right action of $ \cU_{q,\phi}(\hat{u}(2))$ on  $ SU_{q,\phi}(2) $ we were 
	forced to change the order of arguments in pairing and reverse the braiding. It is so because in the braided category each transposition should be replaced by 
	braiding resulting in modified expression often with many new or deformed terms. This situation makes the order of arguments extremely important.
	
	\section{A family of spheres}
	
	In this section we shall discuss quantum spheres as originating from the finite dimensional representation  of $\cU_{q,\phi}(\hat{u}(2))$,
	with a braided module structure in terms of definition \ref{Ract}. Throughout section we use the standard notation of $q$ calculus,
	\begin{equation*}
		\begin{aligned}
			& [m]_q = \frac{q^{m} - q^{-m}}{q - q^{-1}}, &~~ m \in \fZ, \\
			& [m]_q ! = [1]_q[2]_q\cdots[m]_q, &~~  m \in \fN. \\
		\end{aligned}
	\end{equation*}
	
	\subsection{Finite dimensional representations of $\cU_{q,\phi}(\hat{u}(2))$.}
	
	\begin{prop}\label{prop1}
		Let $l \in \frac{1}{2}\fN$ be a half integer, $\psi \in [0, 2 \pi)$, an arbitrary  phase, $V_l$ a $2l+1$-dimensional vector space 
		over $\fC$ with basis $\be_m$, $m = -l, -l + 1, \cdots, l$ and $T_{l, \psi}$ be a map $\cU_{q,\phi}(\hat{u}(2)) \rightarrow GL(V_l)$ 
		defined on generators by:
		\begin{equation}
			\label{rep0}
			\begin{aligned}
				&T_{l, \psi}(k) \be_m = q^m e^{-4 i m\phi - i\psi} \be_m, \qquad
				& T_{l, \psi}(f) \be_m = \pm \sqrt{[l-m]_q[l + m + 1]_q}\be_{m+1},  \\
				& T_{l, \psi}(k^*) \be_m =  q^m e^{4 i m\phi + i\psi} \be_m, 
				& T_{l, \psi}(e) \be_m=  \pm \sqrt{[l+m]_q[l - m + 1]_q}\be_{m-1}, \\
			\end{aligned}
		\end{equation}
		where the signs in $T_{l. \psi}(f)$ and $T_{l, \psi}(e)$ actions are the same.
		Then $T_{l, \psi}$ is irreducible finite dimensional representation of the 
		algebra $\cU_{q,\phi}(\hat{u}(2))$.
	\end{prop}
	The proof is straightforward and we skip it. Note that for $\phi=0$ and $\psi=0$ the representation reduces to  the usual
	representation of $\cU_{q}(su(2))$. However, the additional phase $\psi$, which a priori can be included in the definition of
	the generator $k$ matters, because the map $k \to e^{i\psi} k$ is only an algebra morphism and does not preserve the coalgebra
	structure. To prove the general form of irreducible representations of $U_{q, \phi}(\hat{u}(2))$ one can use similar arguments as in the
	non-braided case of $U_{q}(su(2))$ \cite{Hopf1}.
	
	\begin{thm}
		Assume that $qe^{-4 i \phi}$ is not a root of unity \textsc{(}i.e $(q e^{-4 i \phi})^r \ne 1$ for any $r \in \fN$\textsc{)}, then any finite dimensional irreducible representation $T$ of ~ $\cU_{q,\phi}(\hat{u}(2))$ algebra is equivalent to some of representations $T_{l, \psi}$ defined in (\ref{prop1})
	\end{thm}
	
	\begin{proof}
		As the underlying vector space is finite-dimensional and $k$ is normal, the operator $T(k)$ has at least one eigenvector $\be$ and some corresponding eigenvalue $\mu$. Because
		\begin{equation*}
			T(k)T(f^r)\be = q^r e^{-4r i \phi} T(f^r)T(k) \be = q^r e^{-4r i \phi} \mu T(f^r) \be ,
		\end{equation*} 
		for any $r \in \fN$, so $T(f^r)\be$ is either $0$ or an eigenvector of $T(k)$ to the eigenvalue $q^r e^{-4r i \phi} \mu$.
		
		Because $qe^{-4 i \phi}$ is not the root of unity, $q^r e^{-4r i \phi} \mu \ne \mu$ therefore $\be$ and $T(f^r)\be$ are different eigenvectors for different $r$. Moreover since $V$ is finite dimensional, there exists $r$ such that $T(f^r)\be \ne 0$ and $T(f^{r+1}) \be = 0$.
		
		Given this $r$ let us define basis vector $\be_0 = T(f^r)\be$ with corresponding eigenvalue $\lambda$: $T(k)\be_0 = \lambda \be_0$. Because $k$ and $k^*$ are conjugate, their representations are Hermitian adjoint, so $\be_0$ is also eigenvector of $T(k^*)$ to eigenvalue $\bar{\lambda}$.
		
		By defining $e^{(s)} = e^s/[s]_q!$,  $~~\be_s = T(e^{(s)})\be_0$  one gets:
		\begin{equation*}
			T(k)\be_s = T(k)T(e^{(s)})\be_0 = \lambda q^{-s}e^{4s i \phi }T(e^{(s)})\be_0 = \lambda q^{-s}e^{4s i \phi }\be_s.
		\end{equation*}
		Once again, because $V$ is finite dimensional, there exist the smallest $p \in \fN$ such that $T(e^{(p)}) \be_0 \ne 0$ and $T(e^{(p + 1)}) \be_0 = 0$. For any $r \in \{1, \cdots, p\}$:
		
		\begin{equation*}
			\begin{aligned}
				& T(f)\be_r = T(f)T(e^{(r)}) \be_0  = T\left(e^{(r-1)}\frac{q^{-(r-1)} kk^* - q^{r-1}k^{-1}k^{*-1} }{q - q^{-1}} + e^{(r)}f \right) \be_0 = \\
				& =  T(e^{(r-1)})\frac{q^{-(r-1)} T(k)T(k^*) - q^{r-1}T(k^{-1})T(k^{*-1}) }{q - q^{-1}} \be_0 + T(e^{(r)})T(f) \be_0 = \\
				& = \frac{q^{-r + 1}|\lambda|^2  - q^{r-1} |\lambda|^{-2}}{q - q^{-1}} \be_{r-1}. \\
			\end{aligned}
		\end{equation*}
		Thus, the subspace $V'$ spanned by the vectors $\be_0, \be_1, \cdots , \be_p$ is invariant under the operators $T(e)$, $T(f)$, $T(k)$, $T(k^*)$. Since the representation is irreducible, $V' = V$ and vectors $\be_0, \be_1, \cdots, \be_p$ form a basis of $V$.
		
		The next step is to find all  possible values of $\lambda$. Since $T(e^{(p + 1)}) \be_0 = 0$ and $T(f)\be_0 = 0$ we have:
		
		\begin{equation*}
			\left( T(e^{(p + 1)})T(f)  - T(f)T(e^{(p + 1)}) \right)\be_0 = \frac{q^{-p}|\lambda|^2 - q^p|\lambda|^{-2}}{q - q^{-1}}\be_p = 0.
		\end{equation*}
		Hence $|\lambda|^2 = q^p$ so $\lambda = q^{p/2}e^{ - i \psi}$ where $\psi$ is independent phase. Therefore the representation $T$ is given by:
		
		\begin{equation}
			\label{rep1}
			\begin{aligned}
				& T(f) \be_{r} = [p- r +1]_q \be_{r-1},  &&~~&& T(e)\be_{r} = [r + 1]_q \be_{r+1}, \\
				& T(k) \be_{r} = q^{p/2 - r}e^{4 r i \phi - i \psi} \be_{r}, &&~~&&  T(k^*) \be_{r} = q^{p/2 - r}e^{-4 r i \phi + i \psi} \be_{r}.\\
			\end{aligned}
		\end{equation}
		If one sets $l = \frac{p}{2}$ and $\be_r' = (\pm 1)^{-r} ([l - r]! /[l + r + 1]!)^{1/2} ~\be_{l-r}$ for $r = -l, \cdots, l$ and $\psi' = \psi - 4 l \phi$ then, the formulas (\ref{rep1}) might be rewritten as:
		
		\begin{equation}
			\label{rep2}
			\begin{aligned}
				& T(f) \be_{s}' = \pm\sqrt{[l-s]_q[l + s +1]_q} \be_{s+ 1}',  &&&& T(e) \be_{s}' = \pm\sqrt{[l+s]_q[l - s +1]_q} \be_{s - 1}', \\
				& T(k) \be_s' = q^{s} e^{-4 i s \phi - i\psi' } \be_s', &&&& T(k^*) \be_s' = q^{s} e^{4 i s \phi + i\psi' } \be_s'. \\
			\end{aligned}
		\end{equation}
		This ends the proof.
	\end{proof}
	
	\subsection{The quantum sphere algebras}
	
	Knowing that every irreducible representation of $\cU_{q,\phi}(\hat{u}(2))$ is equivalent to the one described in (\ref{prop1}), in order to 
	find the most general description of quantum spheres with 3 generators, we shall look for  $\cU_{q,\phi}(\hat{u}(2))$ braided module
	algebra, which is generated by fundamental representation ($l=1$) of the\; $\cU_{q,\phi}(\hat{u}(2))$ braided Hopf algebra. This means
	that the algebra is generated by $\be_{-1}, \be_0, \be_1 \in V_1$ and the action of the $\cU_{q,\phi}(\hat{u}(2))$ algebra on these elements
	is given exactly by \eqref{rep2}.
	
	The goal now is to find commutation relations for basis vectors of $V_1$, and a braided action $\Xi$ between $\cU_{q,\phi}(\hat{u}(2))$ 
	and $V_l$, such  that $\cU_{q,\phi}(\hat{u}(2))$ acts covariantly on the algebra generated by those vectors. We assume that the braided 
	action $\Xi$ depends on the phase factor, similarly as in the case of $SU_{q,\phi}(2)$, through an additive index with the generators
	being homogeneous elements: $\delta(\be_k) = lk$, where $l$ is some unknown constant. A further assumption is that the 
	commutation relations can be expressed by polynomials of degrees at most two in basis vectors $\be_i$. This assumption is 
	imposed so as to make new algebras $S_{q,\phi}^2$ a suitable braided generalization of the known structures such as Podle\'s spheres 
	\cite{podles} \cite{podles0}.
	
	To simplify the notation, we write $T_1(h)\be_i = h \la \be_i$ for any $h \in \cU_{q,\phi}(\hat{u}(2))$ 
	and $\be_i \in S_{q,\phi}^2$. Then the explicit form of the representation on generators is given by:
	
	\begin{equation}
		\label{sRep02}
		\begin{aligned}
			& k \la\be_{-1} = q^{-1}e^{4 i \phi - i \psi} \be_{-1}, \quad
			&& k \la \be_0 = e^{ - i \psi}\be_0, \quad
			&&k \la \be_1 = q e^{-4 i \phi - i \psi} \be_1, \\
			& k^* \la \be_{-1} = q^{-1}e^{-4 i \phi + i \psi}\be_{-1}, 
			&& k^* \la \be_0 = e^{i \psi}\be_0, 
			&& k^* \la \be_1 = q e^{4 i \phi + i \psi} \be_1, \\
			& f \la \be_{-1} = \sqrt{[2]_q}\be_0 ,
			&& f \la \be_0 = \sqrt{[2]_q}\be_1, 
			&&f \la \be_1 = 0, \\
			& e \la \be_{-1} = 0,
			&& e \la \be_0 = \sqrt{[2]_q}\be_{-1}, 
			&& e \la \be_1 = \sqrt{[2]_q} \be_0. 
		\end{aligned}
	\end{equation}
	
	Note that by acting $k$ on second degree monomials in eigenvectors one gets
	\begin{equation*}
		k \la (\be_j \be_k) = (k \la \be_j)(k \la \be_k) =  q^{i + j} e^{-4(j + k) i \phi - 2 i \psi}\be_j \be_k
	\end{equation*}
	Therefore, if free phase $\psi \not= 0$, the relations need to be homogeneous polynomials and due to an arbitrary $\phi$,
	the sum of indices $j+k$ must be constant. 
	
	\begin{lem}
		The commutation relations which are compatible with the action of the 
		$\cU_{q,\phi}(\hat{u}(2))$ algebra are, for $\psi\not=0$,
		\begin{equation}
			\label{cs22}
			\begin{aligned}
				& \be_0 \be_1 - q^2e^{2 i \psi} \be_1 \be_0 = 0, \\
				& qe^{-4 i \phi - i \psi} \be_{-1} \be_1 + (e^{i \psi} - q^2e^{ i \psi})\be_0 \be_0 - qe^{-4 i \phi + 3 i \psi}\be_1 \be_{-1} = 0, \\
				& \be_{-1} \be_0 - q^2e^{2 i \psi} \be_0 \be_{-1} = 0, \\
				& -q^{-1} e^{-4 i\phi - 2 i \psi} \be_{-1} \be_1 + \be_0 \be_0 -  q e^{-4 i \phi+ 2 i \psi} \be_1 \be_{-1} = 0. \\
			\end{aligned}    
		\end{equation}
	\end{lem}
	
	\begin{proof}
	 Using the argument above we see that all possible relations can be expressed as:
		
		\begin{equation}
			\label{cs21}
			\begin{aligned}
				& \alpha_1 \be_0 \be_1 + \alpha_2 \be_1 \be_0 = 0, \\
				& \beta_1 \be_{-1} \be_1 + \beta_2 \be_0 \be_0  + \beta_3 \be_1 \be_{-1} = 0, \\
				& \gamma_1 \be_{-1} \be_0 + \gamma_2 \be_0 \be_{-1} = 0, \\
				& \delta_1 \be_{-1} \be_1 + \delta_2 \be_0 \be_0 + \delta_3 \be_1 \be_{-1} = 0. \\
			\end{aligned}
		\end{equation}
		where all $\alpha$'s, $\beta$'s, $\gamma$'s, $\delta$'s are unknown constants.  The first three equations in \eqref{cs21} are 
		connected by the action of $f$ and $e$, for example, $f$ annihilates the first one, and the action of $e$ on it should give the 
		second one. The fourth identity, which is similar to the second one, differs by the assumption that it is annihilated by both 
		$e$ and $f$. To fix the coefficients we precisely examine the action of $f$ and $e$ on each of them, and a simple computation
		demonstrates that the relations (\ref{cs22}) provide a unique nontrivial solution for all coefficients.
		Moreover, such solution exists only for $l = 2$, so the algebra generators are homogeneous with degrees $\delta(\be_k) = 2k$
	\end{proof}

	The above form of relations can be further simplified. By substitution 
	\begin{equation}
		\label{tr2.1}
		\begin{aligned}
			&\be_1 = - e^{2 i \phi} q \sqrt{1 + q^2}^{-1} \be'_1, && &&
			\be_{-1} = e^{2 i \phi} \sqrt{1 + q^2}^{-1} \be'_{-1} \\
		\end{aligned}
	\end{equation}
	one can eliminate parameter $\phi$ from relations (\ref{cs22}) and reduce them further to:
	\begin{equation}
		\label{psinz1}
		\begin{aligned}
			& \be_1' \be_0 = q^{-2} e^{-2 i \psi} \be_0 \be_1' \qquad \qquad
			&& \be_{-1}' \be_0 = q^{2} e^{2 i \psi} \be_0 \be_{-1}' \\
			& \be_1' \be_{-1}' = -q^{-2} e^{-2 i \psi} \be_0 \be_0,
			&& \be_{-1}' \be_1' = -q^{2} e^{2 i \psi} \be_0 \be_0. \\
		\end{aligned}
	\end{equation}
	
	One can easily see that the basis elements of the above algebra are of the form $(e_0)^n (e_1)^m$ and $(e_0)^n (e_{-1})^m$. Note that all commutation relations depend	exclusively on $z=q^2e^{2i\psi}$ or its inverse and therefore the algebra cannot 
	be a star algebra unless $q^2=1$.  Furthermore, we can easily check the representations of this algebra. 
	
	\begin{lem}There are no bounded representations of algebra defined by \eqref{psinz1} such that $\be_0$ possess at least one eigenvector to nonzero eigenvalue.
	\end{lem}
	
	\begin{proof}
		Assume that  $\be_0$ possesses at least one eigenvector $v_0 \in V$ with some eigenvalue $\lambda \neq 0$.
		Let us define:
		\begin{equation*}
			\begin{aligned}
				& v_k = \be'^k_1 v_0 \hspace{0.2 cm,} && v_{-k} = \be'^k_{-1} v_0 \hspace{0.2 cm ,} \\
			\end{aligned}
		\end{equation*}
		For any $k \in \fN$.
		Then clearly $v_k$ and $v_{-k}$ are eigenvectors of $\be_0$ to the eigenvalues $q^{2k} e^{2 k i \psi} \lambda$ and $q^{-2k} e^{-2 k i \psi} \lambda$. For example for $e_1$:
		
		\begin{equation*}
			\be_0 v_1 = \be_0 \be'_1 v_0 = q^2 e^{2 i \psi} \be'_1 \be_0 v_0 = q^2 e^{2 i \psi} \be'_1  \lambda v_0 = q^2 e^{2 i \psi} \lambda v_1 \hspace{0.2 cm.} 
		\end{equation*}
		
		If the representation of $\be_0$ is bounded, then there exist some $N\in \fRr_+$ such that $q^{2k} < N$ $\forall k \in \fN$. Therefore $q < 1$ so $q^{-1} > 1$ and $\lim_{k \to \infty} q^{-2k} = \infty$.
	\end{proof}
	
    Therefore, we are left with the case $\psi = 0$, which is much more interesting.
	
	\begin{lem}
		The quadratic algebra $S_{q,\phi}^2$, which is $\cU_{q,\phi}(\hat{u}(2))$   braided  module algebra has the form,
		
		\begin{equation}
			\label{cs12}
			\begin{aligned}
				& \be_0 \be_1 - q^2 \be_1 \be_0 = \lambda' \be_1, \\
				& qe^{-4 i \phi} \be_{-1} \be_1 + (1 - q^2)\be_0 \be_0 - qe^{-4 i \phi}\be_1 \be_{-1} = \lambda' \be_0, \\
				& \be_{-1} \be_0 - q^2 \be_0 \be_{-1} = \lambda' \be_{-1}, \\
				& -q^{-1} e^{-4 i\phi} \be_{-1} \be_1 + \be_0 \be_0 -  q e^{-4 i \phi} \be_1 \be_{-1} = \rho' 1.
			\end{aligned}    
		\end{equation}
	\end{lem}	
	
	\begin{proof}
		Because each relation should be invariant under the action of $k$ and $k^*$, similarly 
		as in case of $\psi \neq 0$, we first see that all possible relations restrict to:
		\begin{equation}
			\label{cs11}
			\begin{aligned}
				& \alpha_1 \be_0 \be_1 + \alpha_2 \be_1 \be_0 = \lambda_1 \be_1, \\
				& \beta_1 \be_{-1} \be_1 + \beta_2 \be_0 \be_0  + \beta_3 \be_1 \be_{-1} = \lambda_2 \be_0, \\
				& \gamma_1 \be_{-1} \be_0 + \gamma_2 \be_0 \be_{-1} = \lambda_3 \be_{-1}, \\
				& \delta_1 \be_{-1} \be_1 + \delta_2 \be_0 \be_0 + \delta_3 \be_1 \be_{-1} = \rho 1 ,
			\end{aligned}
		\end{equation}
		where all $\alpha$'s, $\beta$'s, $\gamma$'s, $\delta$'s, $\lambda$'s and $\rho$ are unknown constants. 
		Even though  these relations are different than the one for $\psi \neq 0$ case, the argument with raising and lowering indexes still holds, so the first 
		three relations have to be connected. Moreover, because the identity $1$ can be treated as a $1$-dimensional representation, both 
		$f$ and $e$ should annihilate the fourth relation. To prove the statement one has to examine relations discussed above by calculating the action of $e$ and $f$ on them. Furthermore, during calculation, one gets the constraint $l = 2$, so algebra generators are homogeneous with degrees 
		$\delta(\be_k) = 2k$.
	\end{proof}
	
	One can check by a straightforward computation, that the case $\psi = 0 $ after the same substitution as before:
	\begin{equation}
		\label{tr2.1_copy}
		\begin{aligned}
			&\be_1 = - e^{2 i \phi} q \sqrt{1 + q^2}^{-1} \be'_1, && &&
			\be_{-1} = e^{2 i \phi} \sqrt{1 + q^2}^{-1} \be'_{-1} \\
			& \lambda = (1 + q^2)\lambda', && && \rho = (1 + q^2) \rho' \\ 
		\end{aligned}
	\end{equation}
	corresponds to the Podle\'s spheres described in \cite{podles}:
	\begin{equation}
		\label{Po}
		\begin{aligned}
			& (1 + q^2) \be_0 \be'_1 - q^2(1 + q^2) \be'_1 \be_0 = \lambda \be'_1, \\
			& q^2(\be'_1 \be'_{-1} - \be'_{-1}\be'_1) + (1 - q^4)\be_0 \be_0 = \lambda \be_0 ,\\
			& (1 + q^2)\be'_{-1}\be_0 - q^2(1+ q^2)\be_0 \be'_{-1} = \lambda \be'_{-1} ,\\
			& \be'_{-1}\be'_1 + (1 + q^2)\be_0 \be_0 + q^2 \be'_1 \be'_{-1} = \rho 1.\\
		\end{aligned}
	\end{equation}
	Observe that this algebra admit a $*$-structure of the form $\be_i^* = \be_{-i}$ if and only if $\lambda, \rho \in \fRr$.
	In this \eqref{Po} basis the left action of the $\cU_{q,\phi}(\hat{u}(2))$ braided Hopf algebra takes a form:
	
	\begin{equation}
		\label{sRep1}
		\begin{aligned}
			&k \la \be'_1 \!=\! q e^{-4 i \phi} \be'_1, && k \la \be'_{-1} \!=\! q^{-1}e^{4 i \phi} \be'_{-1}, \\
			&k^* \la \be'_1 \!=\! q e^{4 i \phi} \be'_1,  && k^* \la\be'_{-1} \!=\! q^{-1}e^{- 4 i \phi}\be'_{-1}, \\
			& f \la \be'_1 \!=\! 0, &&  f \la \be'_{-1} = e^{-2 i \phi} q^{-\frac{1}{2}}(1 + q^2) \be_0 , \\
			& e \la \be'_1 \!=\! -e^{-2 i \phi} q^{-\frac{3}{2}}(1 + q^2) \be_0, \quad  && e \la \be'_{-1} \!=\! 0, \\
	 &k\la \be_0 \!=\! \be_0, &&  k^* \la \be_0 \!=\! \be_0, \\
	 &f \la \be_0 = - e^{2 i \phi} q^{\frac{1}{2}} \be'_1, \quad  && e \la \be_{0} \!=\! e^{2 i \phi} q^{-\frac{1}{2}} \be'_{-1}.
	\end{aligned}
	\end{equation}
	
	An additional equivalent description of this family of Podle\'s spheres (which uses a different set of generators) may be found in the Appendix \ref{podrep}.
	
	For non-braided algebras $\cU_q(su(2))$ and $SU_{q}(2)$ Podle\'s spheres might be found as subalgebras of $SU_{q}(2)$ annihilated by some self-adjoint, quasi primitive $x \in \cU_q(su(2))$, that is
	\begin{equation*}
		\Delta x  = x \otimes h_1 + h_2 \otimes x
	\end{equation*}
	for some $h_1,h_2 \in \cU_q(su(2))$. Indeed subspace of $SU_q(2)$ annihilated by such $x$ is a subalgebra, since if $x \la a = 0$ and $x \la b = 0$ for some $a,b \in SU_{q}(2)$, then 
	
	\begin{equation*}
		x \la (ab) = (x \la a)(h_1 \la b) + (h_2 \la a)(x \la b) = 0.
	\end{equation*}
	
	Below we generalize this property for braided Hopf algebras.
	
	\begin{thm}
		Let $H$ be a braided Hopf algebra and $A$ an $H$-braided module algebra.  For any element $x \in H$ let us denote 
		the kernel  of its action by $A_x = \{ a \in A: x \la a =0 \}$. If an element $x \in H$ is 
		quasi-primitive, and the braiding $\Xi$ between $H$ and $A$ satisfies
		\begin{equation*}
		   \Xi(x \otimes a) = \tilde{a} \otimes x, \qquad \Xi(h_1 \otimes b) \in A_x \otimes H, 
		\end{equation*}
		for any $a \in A_x, b \in A_x$, then $A_x$ is an algebra. 
	\end{thm}	
	\begin{proof}
		Assume that $a,b \in A_x$. We verify: 
		\begin{equation*}
			\begin{aligned}
				x \la ab 
				&=(\la \cdot \la)(\id \otimes \Xi \otimes \id) \cop x \otimes a \otimes b =\\
				&= \left( x \la \Xi(h_1,a)_{(1)} \right)  \left(  \Xi(h_1,a)_{(2)} \la b \right) 
				+ \left( h_2 \la \Xi(x,a)_{(1)} \right)  \left(  \Xi(x,a)_{(2)} \la b \right) =\\
				& = 0,
			\end{aligned}	
		\end{equation*}
		where the vanishing of both terms follows directly from our assumption.
	\end{proof}	
	
	We can now apply it to the case of $\cU_{q, \phi}(\hat{u}(2))$ and $SU_{q,\phi}(2)$. Note, that
	due to $\phi\not=0$ there is no selfadjoint element of the symmetry algebra that
	satisfies these conditions apart from $x = kk^* -1 $. 
	
	It is easy to see that the subalgebra of $SU_{q, \phi}(2)$, which arises in that way,
	using the right action can be identified with the above spheres, in the following way:
	\begin{equation}
		\label{Po_exm}
		\begin{aligned}
			&\be'_{1} = (1 + q^2)q^{-1}e^{4 i \phi} \gamma^* \alpha^*, &
			\be_{0} = 1 - (1 + q^2)\gamma^* \gamma, \\
			&\be'_{-1} = (1 + q^2)q^{-1}e^{-4 i \phi} \alpha \gamma, 
		\end{aligned}
	\end{equation}
	with $\lambda = 1-q^4$, $\rho = 1 + q^2$ and the $\cU_{q,\phi}(\hat{u}(2))$ action comes from its left action on $SU_{q, \phi}(2)$.
	
	
 \section{Star structure}
	All algebras discussed so far were, in fact, braided Hopf
	$\ast$-algebras or $\ast$-algebras. However, we did not describe yet the compatibility with $\ast$-structure  along the way, because
	the requirement that all braided structures are compatible with the star structure brings a lot of unexpected complications. 
	
	In this section we recall the basic properties of the star structure which plays a significant role in defining the compatibility of the braided module structure with the conjugation. The star structure compatibility for a Hopf algebra action of $H$ on $A$ is defined by the following condition, 
	\begin{equation}
		\label{non_br_sStr}
		(h \la a)^* = (Sh)^* \la a^* , \quad \forall h\in H, a \in A.
	\end{equation}
	It needs to be appropriately translated into the braided case and, in general,
	will require additional compatibility conditions between the star, braiding and
	antipodes. As our interest here lies in the area of specific braiding which is
	obtained by the twisting procedure, we shall restrict our attention to such a case
	only. Nonetheless, we believe that our results may be expanded to a broader framework, leaving a more general approach for future work.
	
	\begin{thm} \label{thm_sStruct}
		Let $H$ be a braided Hopf algebra obtained from twisting  and $A$ be a left braided $H$ module, such that the mutual braiding between $H$ and $A$: $\Xi$ is
		also determined by the degrees $\delta$ of homogeneous elements satisfying condition \eqref{hom1}. 
		If for any homogeneous elements $h \in H$, $a \in A$,
		\begin{equation}
			\delta(h)\delta(a) = \delta(a^*)\delta(h^*),
		\end{equation}
		then left action is $*$-compatible in the sense:
		\begin{equation}
			\label{sStruc_com}
			\hstretch 90  \vstretch 60
			\begin{tangle}
				\object{H} \step[3] \object{A} \\
				\lu[3] \\
				\step[3] \O* \\
				\step[3] \object{A} \\
			\end{tangle}
			\quad = \quad
			\begin{tangle}
				\step[1] \object{H} \step[3] \object{A} \\
				\step[1] \S \step[3] \id \\
				\obox 5* \\
				\step[1] \lu[3] \\
				\step[4] \object{A} \\
			\end{tangle} \quad , 
		\end{equation}
		where the $*$-structure on tensor product $H \otimes A$  is defined by:
		\begin{equation*}
			(h \otimes a)^* = e^{- 2 i \phi \delta(h) \delta(a)} h^* \otimes a^*
		\end{equation*}
		
	\end{thm}
	
	\begin{proof}
		We start by assuming that the conjugation of the tensor product 
		of two homogeneous elements $h \otimes a \in H \otimes A$ acquires a phase
		factor dependent on their degrees,
		\begin{equation}
			(h \otimes a)^* =  \varphi(\delta(h), \delta(a)) \, h^* \otimes a^*,
		\end{equation}
		and we use it in the condition \eqref{sStruc_com} to see what restriction 
		on the function $\varphi$ arises from the compatibility relations. 
		
		To do so, one must check the compatibility between $*$-structure \eqref{sStruc_com} and multiplication both in $H$ and $A$. Because the braiding is defined as transportation times scalar factor and $*$-structure differs from ordinary definition also by some scalar function, it should come as no surprise that the above conditions would give us only the constraints for those scalar terms.
		The full calculations of compatibility conditions are presented in Appendix \ref{sStuc_app} due to their length. The resulting constraints are:
		\begin{equation}
			\label{s_comp1_copy}
			\varphi(\delta(g) + \delta(h), \delta(a)) = \varphi(\delta(h), \delta(g) + \delta(a)) \varphi(g,a) e^{2 i \phi \delta(h) \delta(g)} ~,
		\end{equation}
		\begin{equation}
			\label{s_comp2_copy}
			\varphi(\delta(h), \delta(a) + \delta(b)) = \varphi(\delta(h_{(1)}),\delta(a)) \varphi(\delta(h_{(2)}), \delta(b)) e^{- 2 i \phi \delta(h_{(2)}) \delta(a)} e^{-2 i \phi \delta(h_{(1)}) \delta(b)} ~,
		\end{equation}
		hence, we shall only manipulate \eqref{s_comp1_copy} and \eqref{s_comp2_copy} to derive the general form of $\varphi(~\cdot~ , ~\cdot~)$. First, we may set $\delta(g) = 0$ in \eqref{s_comp1_copy} to obtain:
		\begin{equation*}
			\varphi(\delta(h), \delta(a)) = \varphi(\delta(h), \delta(a))\varphi(0, \delta(a)) ~~,
		\end{equation*}
		Therefore $\varphi(0,\delta(a)) = 1$. Next, consider  \eqref{s_comp2_copy} assuming $\delta(h) = 0$, which implies $\delta(h_{(1)}) + \delta(h_{(2)}) = 0$:
		\begin{equation}
			\label{s_comp3}
			\begin{aligned}
				& 1 = \varphi(\delta(h_{(1)}),\delta(a)) \varphi(-\delta(h_{(1)}),\delta(b)) e^{ 2 i \phi \delta(h_{(1)}) \delta(a)} e^{-2 i \phi \delta(h_{(1)}) \delta(b)} \\
				& \varphi(\delta(h_{(1)}),\delta(a))^{-1}  e^{-  2 i \phi \delta(h_{(1)}) \delta(a)} = \varphi(-\delta(h_{(1)}),\delta(b)) e^{2 i \phi \delta(-h_{(1)}) \delta(b)} \\
			\end{aligned}
		\end{equation}
		
		The left-hand side of \eqref{s_comp3} depends on $\delta(a)$, while right hand side does not. The same holds also for $\delta(b)$ on the right-hand side of \eqref{s_comp3}, hence both sides are some function $f$ of $\delta(h_{(1)})$.  By changing notation $h_{(1)} \rightarrow g$, one gets 
		\begin{equation}
		\label{s_comp3.5}
			\varphi(\delta(g),\delta(a))^{-1}  e^{-  2 i \phi \delta(g) \delta(a)}  = f(\delta(g)),
		\end{equation}
		with $f(0) = 1$. Substitution of \eqref{s_comp3.5} back into equation \eqref{s_comp2_copy} after short manipulation gives: 
		\begin{equation}
			\label{s_comp4}
			\varphi(\delta(h), \delta(a) + \delta(b)) e^{2 i \phi \delta(h) \delta(a)} e^{2 i \phi \delta(h) \delta(b)} = f(\delta(h_{(1)}))^{-1}f(\delta(h) - \delta(h_{(1)}))^{-1}
		\end{equation}
        Similarly as before, both sides are dependent only on $\delta(h)$ since this is the only variable present on both sides. Let us thus rewrite the left hand side of \eqref{s_comp4} as $F(\delta(h))$. This gives us following condition for $f(x)$:
		\begin{equation*}
			1 = F(y) f(x) f(y-x).
		\end{equation*}
		If one sets $x = y$, then above simplifies to $F(x) = f(x)^{-1}$, so in general:
		\begin{equation*}
			f(y)f(x)^{-1} = f(y-x).
		\end{equation*}
		Therefore $f(x)$ is just an exponent function $f(x) = e^{\alpha x}$ with some independent coefficient $\alpha$.  Hence the most general form of a braiding factor $\varphi$ is: 
		\begin{equation}
			\label{s_comp5}
			\varphi(\delta(g), \delta(a)) = e^{\alpha \delta(g)} e^{-2 i \phi\delta(g) \delta(a)}
		\end{equation}
		If algebra $A$ is unital, then by \ref{s_comp5} we have: 
		\begin{equation*}
			\la \circ  ( S(g) \otimes 1)^* = e^{\alpha \delta(g)} S(g)^* \la 1^* =   e^{\alpha \delta(g)} 1(S(g)^*) =   e^{\alpha \delta(g)} 1(S(g))^* = e^{\alpha \delta(g)} \epsilon(g)^*,
		\end{equation*}
		but on the other hand, it must be equal to:
		\begin{equation*}
			(g \la 1)^* = \epsilon(g)^*.
		\end{equation*}
		Therefore $\alpha$ has to be equal to $0$.
	\end{proof}
	
	One may check by simple calculation on generators that action of $\cU_{q,\phi}(\hat{u}(2))$ on itself via adjoin action, on  $SU_{q,\phi}(2)$ and on $S_{q,\phi}^2$ is $*$-compatible in the sense of the definition \eqref{sStruc_com}.

	\section{Conclusions and outlook}
	
	The presented work explicitly constructs the symmetries, understood as braided Hopf algebra actions, of the $SU_{q,\phi}(2)$ algebra and the Podle\'s spheres. 
	To provide this concrete answer we first systematically developed the theory of left action of braided Hopf algebras on algebras. Then we illustrated it with the adjoint action of a braided Hopf algebra on itself and an action obtained from the pairing between two braided Hopf algebras. We further developed the notion of compatibility of the star structure with the action of Hopf algebras on algebras 
	in a braided case, considering the braiding obtained by the twisting procedure.
	
	Subsequently we introduced a new braided Hopf algebra, $\cU_{q,\phi}(\hat{u}(2))$, as an
	analogue of $\cU_{q}(su(2))$ algebra preserving its duality to braided $SU_{q, \phi}(2)$ algebra. 
	
	By using $\cU_{q,\phi}(\hat{u}(2))$ we could provide an explicit example of braided left action obtained from the pairing with $SU_{q, \phi}(2)$ and, by reversing the braiding, the right action.
	This has not been known so far and is a crucial step into a systematic application of the discovered action together with the representation theory 
	of the $\cU_{q,\phi}(\hat{u}(2))$ algebra to develop the spectral geometry
	of the braided $SU_{q, \phi}(2)$ algebra. Moreover, the  presented derivation of the algebras
	of Podle\'s spheres is a different one and, in a sense, complementary to 
	the one presented in \cite{podles}. 
	
	The concrete realisation of the braided Hopf algebra acting on them shows, that one needs to pass to a larger object to accommodate
	for additional constraints imposed by the braiding. Thus, the Podle\'s spheres
	allow, as symmetries in the sense of an action, not only the usual $\cU_{q}(su(2))$ Hopf algebra but also a larger braided Hopf algebra $\cU_{q, \phi}(\hat{u}(2))$.
	
	There are several interesting directions to follow in future research. First of all,
	out of the family of Podle\'s spheres, only the standard one can be realized
	as a subalgebra of the braided $SU_{q,\phi}(2)$ algebra.  This, together with the $q \to1$ limit in the generic case, or $e^{4 i \phi}$ being a root of unity, shall be studied in more detail.
	
	Finally, having a braided analogue of $\cU_q(su(2))$ it is a challenging aim
	to developed theory of spectral triples for $SU_{q, \phi}(2)$ in analogy to \cite{dirac} and for  Podle\'s spheres in braided category in analogy to \cite{dirac_all_podles}.

	\newpage
	\appendix 
	\section{Appendix}
	
	\subsection{Braided Antipode}\label{branti}
	
	In this Appendix we recall the properties of the braided antipode, which are used in our work. For the ordinary Hopf algebra, antipode is an algebra homomorphism from Hopf algebra $H$ to algebra with reversed multiplication $H^{op}$: 
	\begin{equation}
		\label{non_br_anty_m}
		S(gh) = S(h)S(g)
	\end{equation}
	Following \cite{brGeo} we generalize this property for the braided Hopf algebra.
	\begin{equation*}
		\hstretch 90  \vstretch 60
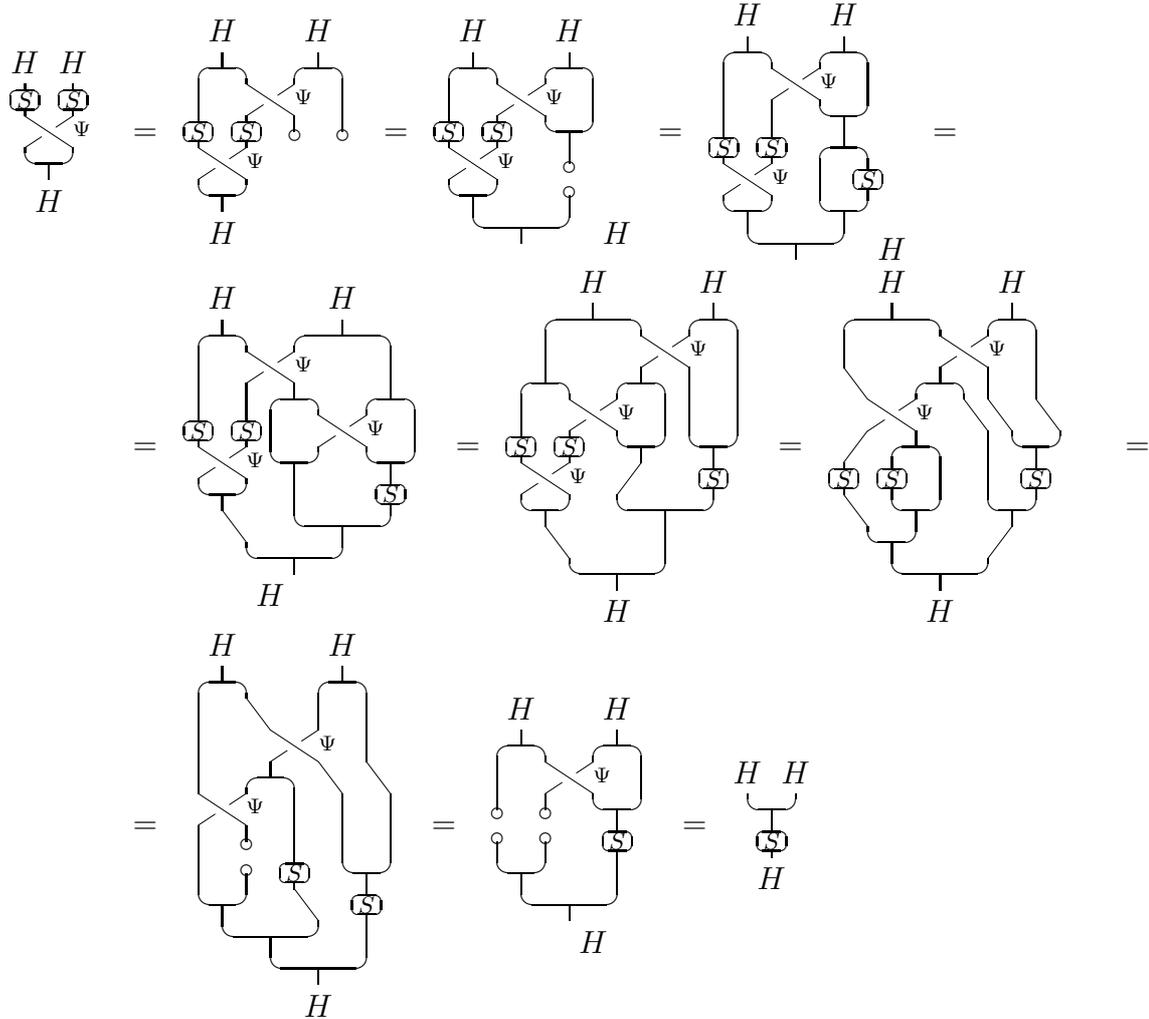
		\begin{aligned}
			\begin{tangle}
				\object{H} \step[2] \object{H} \\
				\S \step[2] \S \\
				\x {{{}^\Psi}} \\
				\cu \\
				\step[1] \object{H}
			\end{tangle}
			\quad & = \quad
			\begin{tangle}
				\step[1] \object{H} \step[4] \object{H} \\
				\cd \step[2] \cd \\
				\id \step[2] \x {{{}^\Psi}} \step[1.21] \id \\
				\S \step[2] \S \step[2] \counit \step[2] \counit \\
				\x {{{}^\Psi}} \\
				\cu \\    
				\step[1] \object{H}
			\end{tangle}
			\quad = \quad
			\begin{tangle}
				\step[1] \object{H} \step[4] \object{H} \\
				\cd \step[2] \cd \\
				\id \step[2] \x {{{}^\Psi}} \step[1.21] \id \\
				\S \step[2] \S \step[2] \cu \\
				\x {{{}^\Psi}} \step[2.21] \counit\\
				\cu  \step[3] \unit \\
				\step[1] \Cu \\
				\step[3] \object{H}
			\end{tangle}
			\quad = \quad
			\begin{tangle}
				\step[1] \object{H} \step[4] \object{H} \\
				\cd \step[2] \cd \\
				\id \step[2] \x {{{}^\Psi}} \step[1.21] \id \\
				\id \step[2] \id \step[2] \cu \\
				\S \step[2] \S \step[2] \cd \\
				\x {{{}^\Psi}} \step[1.21] \id \step[2] \S \\
				\cu  \step[2] \cu \\
				\step[1] \Cu \\
				\step[3] \object{H}
			\end{tangle}
			\quad  =\\
			&= \quad
			\begin{tangle}
				\step[1] \object{H} \step[5] \object{H} \\
				\cd \step[2] \Cd \\
				\id \step[2] \x {{{}^\Psi}} \step[3.21] \id \\
				\id \step[2] \id \step[1] \cd \step[2] \cd \\
				\S \step[2] \S \step[1] \id \step[2] \x {{{}^\Psi}} \step[1.21] \id \\
				\x {{{}^\Psi}} \step[0.21] \cu \step[2] \cu \\
				\cu  \step[2] \id \step[4] \S \\
				\step[1] \nw1 \step[2] \Cu \\
				\step[2] \Cu \\
				\step[4] \object{H}
			\end{tangle}
			\quad = \quad
			\begin{tangle}
				\step[3] \object{H} \step[5] \object{H} \\
				\step[1] \Cd \step[2] \cd \\
				\step[1] \id \step[4] \x {{{}^\Psi}} \step[1.21] \id \\
				\cd \step[2] \cd \step[1] \id \step[2] \id \\
				\id \step[2] \x {{{}^\Psi}} \step[1.21]  \id \step[1] \id \step[2] \id \\
				\S \step[2] \S \step[2] \cu \step[1] \cu \\
				\x {{{}^\Psi}} \step[1.21] \ne1 \step[3] \S \\
				\cu \step[2] \Cu \\
				\step[1] \nw1 \step[4] \id \\
				\step[2] \Cu \\
				\step[4] \object{H}
			\end{tangle}
			\quad = \quad
			\begin{tangle}
				\step[2] \object{H} \step[5] \object{H} \\
				\Cd \step[2] \cd \\
				\id \step[4] \x {{{}^\Psi}} \step[1.21] \id \\
				\nw1 \step[2] \cd \step[1] \id \step[2] \id \\
				\step[1] \x {{{}^\Psi}} \step[1.21]  \nw1 \nw1 \step[1] \nw1 \\
				\ne1 \step[1] \cd \step[2] \id \step[1] \cu \\
				\S \step[2] \S  \step[2] \id \step[2] \id \step[2] \S \\
				\nw1 \step[1] \cu \step[2] \cu \step[2] \\
				\step[1] \cu \step[3] \ne1 \\
				\step[2] \Cu \\
				\step[4] \object{H}
			\end{tangle}
			\quad =\\
			& = \quad 
			\begin{tangle}
				\step[1] \object{H} \step[5] \object{H} \\
				\cd \step[3] \cd \\
				\id \step[2] \nw1 \step[2] \id \step[2] \id \\
				\id \step[3] \x {{{}^\Psi}} \step[1.21]  \id  \\
				\id \step[2] \cd \step[1] \nw1 \step[1] \nw1 \\ 
				\x {{{}^\Psi}} \step[1.21] \id \step[2] \id  \step[2] \id \\
				\id \step[2] \counit \step[2] \id \step[2] \id \step[2] \id\\
				\S \step[2] \unit \step[2] \id \step[2] \cu \\
				\cu \step[2] \nw1 \step[2] \S \\
				\step[1] \Cu \step[2] \id \\
				\step[3] \Cu \\
				\step[5] \object{H} \\
			\end{tangle}
			\quad  = \quad
			\begin{tangle}
				\step[1] \object{H} \step[4] \object{H} \\
				\cd \step[2] \cd  \\
				\id \step[2] \x {{{}^\Psi}} \step[1.21] \id \\
				\counit \step[2] \counit \step[2] \cu \\
				\unit \step[2] \unit \step[3] \S \\
				\cu \step[3] \id \\
				\step[1] \Cu \\
				\step[3] \object{H} \\
			\end{tangle}
			\quad = \quad
			\begin{tangle}
				\object{H} \step[2] \object{H} \\
				\cu  \\
				\step[1] \S \\
				\step[1] \object{H} \\
			\end{tangle}
		\end{aligned}
	\end{equation*}
	By complete analogy we may show the same property for the antipode and the coproduct, denoted as,
	\begin{equation*}
		\hstretch 90  \vstretch 60
		\begin{tangle}
			\step[1] \object{H} \\
			\cd \\
			\x {{{}^\Psi}} \\
			\S \step[2] \S \\
			\object{H} \step[2] \object{H} \\
		\end{tangle}
		\quad = \quad
		\begin{tangle}
			\step[1] \object{H} \\
			\step[1] \id \\
			\step[1] \S \\ 
			\cd \\
			\object{H} \step[2] \object{H} \\
		\end{tangle}
	\end{equation*}

	\subsection{The classical limit of $\cU_{q, \phi}(\hat{u}(2))$}\label{claslim}
    In this Appendix we consider the $q \to 1$ limit of the algebra  $\cU_{q,\phi}(\hat{u}(2)) $ given by \eqref{eq5}. Firstly let us define $q$ to be $e^h$ for $h<0$ and $H$ to be an operator such that $k=\exp{(h H)}$.
	Then, by the expansion of exponents in the first commutation relation from \eqref{eq5} one obtains:
	\begin{equation*}
		\begin{aligned}
			& ek = q e^{- 4 i \phi} ke \\
			& e(1 + hH + \cdots) = (1 + h + \cdots) e^{-4 i \phi} (1 + h H + \cdots) e \\
			& e + h~ eH + \cdots = e^{-4 i \phi} e + h e^{-4 i \phi}(e + He) +  \cdots \\
		\end{aligned}
	\end{equation*}
	
	Comparing the $0^{th}$ order in $h$ we come to $e^{-4 i \phi} = 1$, which is true only if $\phi$ is a multiple of $\pi/2$. Thus classical limit $h \to 0$ cannot be obtained without involving the limit in phase $\phi$ as well. Therefore we redefine $e^h =q e^{-4 i \phi}$, $h = \ln(q) - 4 i \phi$, so limit $h \to 0 $ correspond to $q \to 1$ and $\phi \to 0$ simultaneously. Then the above calculation changes into:
	\begin{equation*}
		\begin{aligned}
			& ek = q e^{- 4 i \phi} ke \\
			& e(1 + hH + \cdots) = (1 + h + \cdots)  (1 + h H + \cdots) e \\
			& e + h ~eH + \cdots =  e + h(e + He) +  \cdots \\
		\end{aligned}
	\end{equation*}
	
	So in the $0^{th}$ order we simply get $e = e$ and the limit $h \to 0$ gives $[e,H] = e$.
	
	The situation looks even more interesting for the commutator $[f,e]$, where up to the first order we obtain:
	\begin{equation}
		\label{com_limit}
		[f,e] = \left( 1 + \frac{ \Im(h)}{\fR(h)}\right) H + \left( 1 -  \frac{ \Im(h)}{\fR(h)}\right) H^*~.
	\end{equation}
	
	Therefore, the classical limit should be taken as $h \to 0$, and $\phi/\ln(q) \to 0$. Then the commutation relations \eqref{eq5} in the limit case are:
	\begin{equation}
		\begin{aligned}
			& [e,H] = e~, && [e, H^*] = e~, &&  [f,e] = H + H^*~, \\
			& [H,f] = f~, && [H^*,f] = f~, && [H,H^*] = 0 ~.\\
		\end{aligned}
	\end{equation}
	
	By taking limits for coproduct in the discussed manner one obtains:
	\begin{equation}
		\begin{aligned}
			& \cop e = e \otimes 1 + 1 \otimes e~, && \cop f = f \otimes 1 + 1 \otimes f~, \\
			& \cop H = H \otimes 1 + 1 \otimes H~, && \cop H^* = H^* \otimes 1 + 1 \otimes H^*~. \\    
		\end{aligned}
	\end{equation}
	
	If we assume that $H$ is hermitian, the resulting bialgebra is $su(2)$ Lie algebra. Otherwise $H$ may be decomposed into hermitian and antihermitian part $H = H_h + H_a$ so that obtained bialgebra is $u(2) = su(2) \oplus u(1)$ with $su(2)$ generated by $e,f,H_h$ and $u(1)$ generated by $H_a$.
	
	Note that one would have greater freedom of taking limits $q \to 1$ and $\phi \to 0$ independently, if one additionally assumes form the beginning that $k$ and $k^*$ are both exponents of hermitian operator $H$, ie;
	\begin{equation*}
		\begin{aligned}
			k = e^{( \log(q) - 4 i \phi) H } && k^* = e^{( \log(q) + 4 i \phi) H }
		\end{aligned}
	\end{equation*}
	
	In that case \eqref{com_limit} simplify to $ [f,e] = 2H$ 
	without additional requirement $\phi/\ln(q) \to 0$ and the limit $q \to 1$, $\phi \to 0$ of algebra $\cU_{q, \phi}(\hat{u}(2))$ simplifies to $\cU(su(2))$.

	\subsection{Left and right action on $SU_{q, \phi}(2)$}\label{appact}
	Below we present the left and the right action of $\cU_{q, \phi}(\hat{u}(2))$ generators on $SU_{q, \phi}(2)$ generators. 
	The explicit form of left action of $\cU_{q, \phi}(\hat{u}(2))$ on $SU_{q, \phi}(2)$ on generators is given by:
	
	\begin{equation}
		\label{eq14}
		\begin{aligned}
			& e \la \alpha = 0,  && e \la  \alpha^* = \gamma, \\
			& e \la \gamma = 0,  && e \la  \gamma^* = - q^{-1} e^{-4 i \phi} \alpha, \\
			& f \la \alpha = -q e^{4 i \phi} \gamma^*, && f \la  \alpha^* = 0, \\ 
			& f \la \gamma = \alpha^*,  && f \la  \gamma^* = 0, \\
			& k \la \alpha = q^{-\frac{1}{2}} e ^{2 i \phi} \alpha, && k \la \alpha^* = q^{\frac{1}{2}} e ^{-2 i \phi} \alpha^*, \\
			& k \la \gamma = q^{-\frac{1}{2}} e ^{2 i \phi} \gamma, && k \la \gamma^* = q^{\frac{1}{2}} e ^{-2 i \phi} \gamma^*, \\
			& k^* \la \alpha = q^{-\frac{1}{2}} e ^{- 2 i \phi} \alpha, && k^* \la \alpha^* = q^{\frac{1}{2}} e^{2 i \phi} \alpha^*, \\ 
			& k^* \la \gamma = q^{-\frac{1}{2}} e ^{-2 i \phi} \gamma, && k^* \la \gamma^* = q^{\frac{1}{2}} e ^{2 i \phi} \gamma^*, \\
			& k^{-1} \la \alpha = q^{\frac{1}{2}} e ^{-2 i \phi} \alpha, && k^{-1} \la \alpha^* = q^{-\frac{1}{2}} e ^{2 i \phi} \alpha^*, \\ 
			& k^{-1} \la \gamma = q^{\frac{1}{2}} e ^{-2 i \phi} \gamma, && k^{-1} \la \gamma^* = q^{-\frac{1}{2}} e ^{2 i \phi} \gamma^*, \\
			& (k^*)^{-1} \la \alpha = q^{\frac{1}{2}} e ^{2 i \phi} \alpha, && (k^*)^{-1} \la \alpha^* = q^{-\frac{1}{2}} e^{-2 i \phi} \alpha^* , \\
			& (k^*)^{-1} \la \gamma = q^{\frac{1}{2}} e ^{2 i \phi} \gamma, \qquad \qquad && (k^*)^{-1} \la \gamma^* = q^{-\frac{1}{2}} e ^{-2 i \phi} \gamma^*, 
		\end{aligned}
	\end{equation}
	while for the right action, we have,
	\begin{equation}
		\label{eq17}
		\begin{aligned}
			& \alpha \ra e = \gamma, && \alpha^* \ra e = 0, \\
			& \gamma  \ra e = 0,  && \gamma^* \ra e = - q^{-1} e^{-4 i \phi} \alpha^*,\\
			& \alpha \ra f = 0, &&  \alpha^* \ra f = -q e^{4 i \phi} \gamma^*, \\
			& \gamma \ra f= \alpha,  && \gamma^* \ra f = 0,\\
			& \alpha \ra k = \sqi e^{2 i \phi} \alpha, && \alpha^* \ra k= \sq  e ^{-2 i \phi} \alpha^*, \\
			& \gamma^* \ra k = \sqi e^{2 i \phi} \gamma^*, && \gamma \ra k  = \sq e ^{-2 i \phi} \gamma, \\
			& \alpha \ra k^* = \sqi e ^{- 2 i \phi} \alpha, \qquad \qquad && \alpha^* \ra k^*  = \sq  e^{2 i \phi} \alpha^*, \\
			& \gamma^* \ra k^*  = \sqi e ^{-2 i \phi} \gamma^* && \gamma \ra k^* = \sq e ^{2 i \phi} \gamma, \\
			& \alpha \ra k^{-1} = \sq e^{-2 i \phi} \alpha, && \alpha^* \ra k^{-1}  = \sqi e ^{2 i \phi} \alpha^*, \\
			& \gamma^* \ra k^{-1} = \sq  e ^{-2 i \phi} \gamma^*, \qquad && \gamma \ra k^{-1} =\sqi e ^{2 i \phi} \gamma, \\
			& \alpha \ra (k^*)^{-1} = \sq  e ^{2 i \phi} \alpha, && \alpha^* \ra (k^*)^{-1}  = \sqi e^{-2 i \phi} \alpha^*, \\
			&\gamma^* \ra (k^*)^{-1}  = q^{1/2} e ^{2 i \phi} \gamma^* \qquad && \gamma \ra (k^*)^{-1}  = \sqi e ^{-2 i \phi} \gamma.
		\end{aligned}
	\end{equation}
	\subsection{Podle\'s spheres representation}\label{podrep}
	
	In this Appendix we rewrite the obtained algebra of the Podle\'s sphere $S_{q,\phi}^2$  \eqref{Po} in a slightly
	different form (used by some authors):
	\begin{equation}
		\begin{aligned}
			& q^2 A b =  b A, \qquad \qquad & b B +  q^4 A^2 +  q^2  A + c= 0, \\
			& A B = q^2 B A, \quad &  B b + A^2 +  A + c = 0.
		\end{aligned}
	\end{equation}
	The $\be_0, \be_1,\be_{-1}$ generators can be expressed as,
	\begin{equation}
		\label{tr1.1}
		\begin{aligned}
			& \be_0 = \alpha^{-1} A + (\alpha (1 + q^2))^{-1}\\
			& \be_{-1} = (1 + q^2)^{-1/2} q^{-1} \rho \alpha^{-1} b \\
			& \be_1 = -(1 + q^2)^{-1/2} \zeta \alpha^{-1} B.
		\end{aligned}
	\end{equation}
	Then the action of $\cU_{q,\phi}(\hat{u}(2))$  on the
	generators $A,B,b$ takes the form:
	\begin{equation}
		\begin{aligned}
			&e \la A = \frac{1}{q} \frac{1}{\sqrt{q}}  \rho b,  \qquad \qquad   &e  \la b = 0, \\
			&e \la B = - (1+q^2) \frac{1}{\sqrt{q}}  e^{-4i\phi} \rho  A +\frac{1}{\sqrt{q}} e^{-4i\phi} \rho, \\
			&f \la A = - \frac{1}{\sqrt{q}}  \zeta B, \qquad \qquad  &f \la B = 0, \\
			&f \la b = (1+q^2) \sqrt{q}  \zeta   e^{-4i\phi} A + \sqrt{q}\zeta   e^{-4i\phi} , 
		\end{aligned}
	\end{equation}
	and, 
	\begin{equation}
		\begin{aligned}
			&k \la A = A, \qquad \qquad     &k \la b = q^{-1} e^{4i\phi} b, \qquad \qquad  &k \la B = q e^{-4i\phi} B, \\
			&k^* \la A = A, \qquad \qquad    &k^* \la b = q^{-1} e^{-4i\phi} b, \qquad \qquad    &k^* \la B = q e^{4i\phi}B, \\
		\end{aligned}
	\end{equation}
	where $\rho \zeta e^{-4i\phi} = 1 $ and the relations between $S_{q,\phi}^2$ the constants are
	\begin{equation}
		\label{tr1.3}
		\begin{aligned}
			& \lambda' = \frac{1-q^2}{1 + q^2} \alpha^{-1}~, \\
			& \rho' = \left(- c q^{-2} + (1 + q^2)^{-2}\right) \alpha^{-2}~. \\ 
		\end{aligned}
	\end{equation}
	Finally, the $*$-structure in this basis takes a form:
	\begin{equation*}
		\label{str2}
		\begin{aligned}
			& B^* = |\zeta|^{-2} b~,~~  &&  A^* = A~,~~ && b^* = |\rho|^{-2} B~. \\
		\end{aligned}
	\end{equation*}
	\subsection{Star structure}\label{sStuc_app}
	
	In the final Appendix, we complete the omitted calculations from the proof of the Theorem \ref{thm_sStruct}:  the compatibility between $*$-structure and associativity both in Hopf Algebra $H$ and its left module $A$. 
	In the computation below we frequently use the fact that $\delta(h \la a) = \delta(h) + \delta(a)$, which is equivalent to the condition \eqref{action_comp} from the definition of covariant braided action for algebras obtained from twisting.
	
	Let us first consider the compatibility of the $*$-structure with associativity in $H$.
	Next to the diagram representing the composition of maps, we write the occurring
	phase factors.  The expression $(hg \la a)^*$ gives:
	\begin{equation*}
		\hstretch 90  \vstretch 60
		\begin{aligned}
			\begin{tangle}
				\object{H} \step[2] \object{H} \step[2] \object{A} \\
				\cu \step[2] \id \\
				\step[1] \lu[3] \\
				\step[4] \O* \\
				\step[4] \object{A} \\
			\end{tangle}
			\quad & = \quad
			\begin{tangle}
				\object{H} \step[2] \object{H} \step[2] \object{A} \\
				\id \step[2] \lu[2] \\
				\lu[4] \\
				\step[4] \O* \\
				\step[4] \object{A} \\
			\end{tangle}
			\quad = \quad
			\begin{tangle}
				\step[1] \object{H} \step[2] \object{H} \step[2] \object{A} \\
				\step[1] \S \step[2] \lu[2] \\
				\obox 6*  \\
				\step[1] \lu[4] \\
				\step[5] \object{A} \\
			\end{tangle}
			\quad = \quad
			\begin{tangle}
				\object{H} \step[2] \object{H} \step[2] \object{A} \\
				\S \step[2] \lu[2] \\
				\O* \step[4] \O* \step[0.5] \cdot{\scriptstyle \varphi(\delta(h), \delta(g \la a))} \vspace{-0.07 cm}\\
				\lu[4] \\
				\step[4] \object{A} \\
			\end{tangle}
			\quad \\
			& = \quad
			\begin{tangle}
				\object{H} \step[3] \object{H} \step[2] \object{A} \\
				\S \step[3] \S \step[2] \id \\
				\O* \step[2] \obox 4* \\
				\id \step[3] \lu[2] \step[0.5] \cdot {\scriptstyle\varphi(\delta(h),\delta(g)+ \delta(a))}  \vspace{-0.07 cm}\\
				\lu[5] \\
				\step[5] \object{A} \\
			\end{tangle}
			\quad = \quad
			\begin{tangle}
				\object{H} \step[2] \object{H} \step[2] \object{A} \\
				\S \step[2] \S \step[2] \id \\
				\O* \step[2] \O* \step[2] \O* \\
				\id \step[2] \lu[2] \step[0.5] \cdot {\scriptstyle\varphi(\delta(g), \delta(a))~\varphi(\delta(h),\delta(g)+ \delta(a))}  \vspace{-0.07 cm}\\
				\lu[4] \\
				\step[4] \object{A} \\
			\end{tangle}
		\end{aligned}
	\end{equation*}
	This, by definition of $*$-structure compatibility, should be equal to $\la \circ \ast (S(hg) \otimes a$):
	
	\begin{equation*}
		\hstretch 90  \vstretch 60
		\begin{aligned}
			\begin{tangle}
				\object{H} \step[2] \object{H} \step[2] \object{A} \\
				\cu \step[2] \id \\
				\step[1] \S \step[3] \id \\
				\obox 5*  \\
				\step[1] \lu[3] \\
				\step[4] \object{A} \\
			\end{tangle}
			\quad & = \quad
			\begin{tangle}
				\object{H} \step[2] \object{H} \step[2] \object{A} \\
				\x {{{}^\Psi}} \step[1.21] \id \\
				\S \step[2] \S \step[2] \id \\
				\cu \step[2] \id \\
				\obox 5* \\
				\step[1] \lu[3] \\
				\step[4] \object{A} \\
			\end{tangle}
			\quad = \quad
			\begin{tangle}
				\object{H} \step[2] \object{H} \step[2] \object{A} \\
				\S \step[2] \S \step[2] \id \\
				\x {{{}^\Psi}} \step[1.21] \id \\
				\cu \step[2] \id \\
				\step[1] \O* \step[3] \O* \step[0.5]\cdot {\scriptstyle \varphi(\delta(S(g)S(h)), \delta(a))}  \vspace{-0.07 cm}\\
				\step[1] \lu[3]   \\
				\step[4] \object{A} \\
			\end{tangle}
			\quad = \quad
			\begin{tangle}
				\object{H} \step[2] \object{H} \step[2] \object{A} \\
				\S \step[2] \S \step[2] \id \\
				\x {{{}^\Psi}} \step[1.21] \id \\
				\xd \step[2] \id \\
				\O* \step[2] \O* \step[2] \O* \step[0.5]\cdot {\scriptstyle \varphi(\delta(g)+ \delta(h), \delta(a))}  \vspace{-0.07 cm}\\
				\cu \step[2] \id \\
				\step[1] \lu[3] \\
				\step[4] \object{A} \\
			\end{tangle} 
			\quad = \\
			&= \quad
			\begin{tangle}
				\object{H} \step[2] \object{H} \step[2] \object{A} \\
				\S \step[2] \S \step[2] \id \\
				\x {{{}^\Psi}} \step[1.21] \id \\
				\xx {{{}^{\Psi^{-1}}}} \step[0.10] \id \step[0.5]\cdot e^{2 i \phi \delta(h)\delta(g)} \\
				\O* \step[2] \O* \step[2] \O* \step[0.5]\cdot {\scriptstyle \varphi(\delta(g)+ \delta(h), \delta(a)) } \vspace{-0.07 cm}\\
				\cu \step[2] \id \\
				\step[1] \lu[3] \\
				\step[4] \object{A} \\
			\end{tangle}
			\quad = \quad
			\begin{tangle}
				\object{H} \step[2] \object{H} \step[2] \object{A} \\
				\S \step[2] \S \step[2] \id \\
				\O* \step[2] \O* \step[2] \O* \step[0.5]\cdot  e^{-2 i \phi \delta(h)\delta(g)}  {\scriptstyle \varphi(\delta(g)+ \delta(h), \delta(a))}  \vspace{-0.07 cm}\\
				\id \step[2] \lu[2] \\
				\lu[4] \\
				\step[4] \object{A} \\
			\end{tangle}
		\end{aligned}
	\end{equation*}
	In the last step the phase in $e^{2 i \phi \delta(h)\delta(g)}$ was changed by a complex conjugation arising from the $*$ operator. Moreover, the crossed dotted lines represent an ordinary transposition (no braiding involved). By comparing the results we see that they are equal to each other if and only if:
	\begin{equation}
		\label{s_comp1}
		\varphi(\delta(g) + \delta(h), \delta(a)) = \varphi(\delta(h), \delta(g) + \delta(a)) \varphi(g,a) e^{2 i \phi \delta(h) \delta(g)} 
	\end{equation}
	
	Similarly we test the compatibility of the definition of $*$- structure with the associativity in $A$ considering expression $[h \la (ab)]^*$: 
	\vspace{0.2 cm}
	\begin{equation*}
		\hstretch 90  \vstretch 60
		\begin{aligned}
			\begin{tangle}
				\object{H} \step[2] \object{A} \step[2] \object{A} \\
				\id  \step[2] \cu \\
				\lu[3] \\
				\step[3] \O* \\
				\step[3] \object{A} \\
			\end{tangle}
			\quad &= \quad
			\begin{tangle}
				\step[1] \object{H} \step[2] \object{A} \step[2] \object{A} \\
				\step[1] \S  \step[2] \cu \\
				\obox 6* \\
				\step[1]\lu[3] \\
				\step[4] \object{A} \\
			\end{tangle}
			\quad = \quad
			\begin{tangle}
				\object{H} \step[2] \object{A} \step[2] \object{A} \\
				\S  \step[2] \cu \\
				\O* \step[3] \O*  \step[0.5] \cdot {\scriptstyle\varphi(\delta(h), \delta(ab))} \vspace{-0.07 cm}\\
				\lu[3] \\
				\step[3] \object{A} \\
			\end{tangle}
			\quad \\
			&= \quad
			\begin{tangle}
				\object{H} \step[2] \object{A} \step[2] \object{A} \\
				\S  \step[2] \xd \\
				\O* \step[2] \O* \step[2] \O* \step[0.5] \cdot {\scriptstyle \varphi(\delta(h), \delta(a) + \delta(b))} \vspace{-0.07 cm}\\
				\id \step[2] \cu  \\
				\lu[3] \\
				\step[3] \object{A} \\
			\end{tangle}
			\quad =  \quad
			\begin{tangle}
				\step[1] \object{H} \step[3] \object{A} \step[2] \object{A} \\
				\step[1] \S  \step[3] \xd \\
				\step[1] \O* \step[3] \O* \step[2] \O* \step[0.5] \cdot {\scriptstyle \varphi(\delta(h), \delta(a) + \delta(b))} \vspace{-0.07 cm}\\
				\cd  \step[2] \id \step[2] \id \\
				\id \step[2] \x {{{}^\Psi}} \step[1.21] \id \\
				\lu[2] \step[2] \lu[2] \\
				\step[2] \Cu \\
				\step[4] \object{A} \\
			\end{tangle}
		\end{aligned}
	\end{equation*}
	But on the other hand, by $*$-structure compatibility it should be equal to:
	\begin{equation*}
		\hstretch 90  \vstretch 60
		\begin{aligned}
			& \begin{tangle}
				\object{H} \step[2] \object{A} \step[2] \object{A} \\
				\id  \step[2] \cu \\
				\lu[3] \\
				\step[3] \O* \\
				\step[3] \object{A} \\
			\end{tangle}
			\quad = \quad
			\begin{tangle}
				\step[1]\object{H} \step[3] \object{A} \step[2] \object{A} \\
				\cd   \step[2] \id \step[2] \id \\
				\id \step[2] \x {{{}^\Psi}} \step[1.21] \id \\
				\lu[2] \step[2] \lu[2]\\
				\step[2] \Cu \\
				\step[4] \O* \\
				\step[4] \object{A} \\
			\end{tangle}
			\quad = \quad
			\begin{tangle}
				\step[1]\object{H} \step[3] \object{A} \step[2] \object{A} \\
				\cd   \step[2] \id \step[2] \id \\
				\id \step[2] \x {{{}^\Psi}} \step[1.21] \id \\
				\lu[2] \step[2] \lu[2]\\
				\step[2] \O* \step[4] \O* \\
				\step[2] \id \step[3] \ne2 \\
				\step[2] \xd \\
				\step[2] \cu \\
				\step[3] \object{A} \\
			\end{tangle}
			\quad = \quad
			\begin{tangle}
				\step[2]\object{H} \step[3] \object{A} \step[2] \object{A} \\
				\step[1] \cd   \step[2] \id \step[2] \id \\
				\step[1]\id \step[2] \x {{{}^\Psi}} \step[1.21] \id \\
				\step[1] \S \step[2] \id \step[2] \S \step[2] \id \\
				\obox4*  \obox4* \\
				\step[1] \lu[2] \step[2] \lu[2]\\
				\step[3] \id \step[3] \ne2 \\
				\step[3] \xd \\
				\step[3] \cu \\
				\step[4] \object{A} \\
			\end{tangle}
			\quad = \\
			& = \quad
			\begin{tangle}
				\step[1]\object{H} \step[3] \object{A} \step[2] \object{A} \\
				\cd   \step[2] \id \step[2] \id \\
				\id \step[2] \x {{{}^\Psi}} \step[1.21] \id \\
				\S \step[2] \id \step[2] \S \step[2] \id \\
				\O* \step[2] \O* \step[2] \O* \step[2]  \O* \step[0.4] \cdot { \scriptstyle \varphi(S(h_{(1)}),a)\varphi(S(h_{(2)}),b) } \vspace{-0.12 cm} \\
				\lu[2] \step[2] \lu[2]\\
				\step[2] \id \step[3] \ne2 \\
				\step[2] \xd \\
				\step[2] \cu \\
				\step[3] \object{A} \\
			\end{tangle}
			\quad = \quad 
			\begin{tangle}
				\step[1]\object{H} \step[3] \object{A} \step[2] \object{A} \\
				\cd   \step[2] \id \step[2] \id \\
				\S \step[2] \S \step[2] \id \step[2] \id \\
				\id \step[2] \x {{{}^\Psi}} \step[1.21] \id \\
				\O* \step[2] \O* \step[2] \O* \step[2]  \O* \step[0.4] \cdot { \scriptstyle \varphi(S(h_{(1)}),a)\varphi(S(h_{(2)}),b) } \vspace{-0.12 cm} \\
				\lu[2] \step[2] \lu[2]\\
				\step[2] \id \step[3] \ne2 \\
				\step[2] \xd \\
				\step[2] \cu \\
				\step[3] \object{A} \\
			\end{tangle}
			\quad = \\
		\end{aligned}
	\end{equation*}
	
	\begin{equation*}
		\hstretch 90  \vstretch 60
		\begin{aligned}    
			\phantom{xxx}  &  = \quad
			\begin{tangle}
				\step[1]\object{H} \step[3] \object{A} \step[2] \object{A} \\
				\step[1] \S \step[3] \id \step[2] \id \\
				\cd   \step[2] \id \step[2] \id \\
				\xx {{{}^{\Psi^{-1}}}} \step[0.10]  \id \step[2] \id \\
				\id \step[2] \x {{{}^\Psi}} \step[1.21] \id \\
				\O* \step[2] \O* \step[2] \O* \step[2]  \O* \step[0.4] \cdot {\scriptstyle  \varphi(S(h_{(1)}),a)\varphi(S(h_{(2)}),b) } \vspace{-0.12 cm} \\
				\lu[2] \step[2] \lu[2]\\
				\step[2] \id \step[3] \ne2 \\
				\step[2] \xd \\
				\step[2] \cu \\
				\step[3] \object{A} \\
			\end{tangle}
			\quad = \quad 
			\begin{tangle}
				\step[1]\object{H} \step[3] \object{A} \step[2] \object{A} \\
				\step[1] \S \step[3] \id \step[2] \id \\
				\cd   \step[2] \id \step[2] \id \\
				\xd \step[2] \id \step[2] \id \step[0.4] \cdot e^{-2 i \phi \delta(S(h_{(2)}) \delta(S(h_{(1)}))} \\
				\id \step[2] \xd \step[2] \id  \step[0.4] \cdot  e^{ 2 i \phi \delta(S(h_{(2)}))\delta(a)} \\
				\O* \step[2] \O* \step[2] \O* \step[2]  \O* \step[0.4] \cdot {\scriptstyle  \varphi(S(h_{(1)}),a)\varphi(S(h_{(2)}),b) } \vspace{-0.12 cm} \\
				\lu[2] \step[2] \lu[2]  \\
				\step[2] \id \step[3] \ne2 \\
				\step[2] \xd \\
				\step[2] \cu \\
				\step[3] \object{A} \\
			\end{tangle}
			\quad = 
		\end{aligned}
	\end{equation*}
	\begin{equation*}
		\hstretch 90  \vstretch 60
		\begin{aligned}               
			\phantom{xxx} & \qquad= 
 \qquad 
			\begin{tangle}
				\step[1]\object{H} \step[3] \object{A} \step[2] \object{A} \\
				\step[1] \S \step[3] \xd \\
				\cd   \step[2] \id \step[2] \id  \step[0.4] \cdot  e^{-2 i \phi \delta(S(h_{(2)}) \delta(S(h_{(1)}))} e^{ 2 i \phi \delta(S(h_{(2)}))\delta(a)} \\
				\O* \step[2] \O* \step[2] \O* \step[2] \O* \step[0.4] \cdot {\scriptstyle  \varphi(S(h_{(1)}),a)\varphi(S(h_{(2)}),b) }\vspace{-0.12 cm}\\
				\id \step[2] \xd \step[2] \id \\
				\lu[2] \step[2] \lu[2]\\
				\step[2] \Cu \\
				\step[4] \object{A} \\
			\end{tangle}
			\end{aligned}
	\end{equation*}    
	\begin{equation*}
	\hstretch 90  \vstretch 60
	\begin{aligned}
		\qquad &= 
 \qquad 
  	\begin{tangle}
				\step[1]\object{H} \step[3] \object{A} \step[2] \object{A} \\
				\step[1] \S \step[3] \xd \\
				\cd   \step[2] \id \step[2] \id  \\
				\O* \step[2] \O* \step[2] \O* \step[2] \O* \step[0.4] \cdot {\scriptstyle  \varphi(S(h_{(1)}),a)\varphi(S(h_{(2)}),b) } \vspace{-0.12 cm} \\
				\id \step[2] \x {{{}^\Psi}} \step[1.21] \id \\
				\lu[2] \step[2] \lu[2] \step[0.4] \cdot e^{-2 i \phi\delta(S(h_{(1)})^*) \delta(b^{*})}  \\
				\step[2] \Cu \step[0.4] \cdot   e^{- 2 i \phi \delta(S(h_{(2)}))\delta(a)}\\
				\step[4] \object{A}  \step[2.4] \cdot e^{2 i \phi \delta(S(h_{(2)}) \delta(S(h_{(1)}))} \\
			\end{tangle}
		\qquad &= 
 \qquad 
			\begin{tangle}
				\step[1]\object{H} \step[3] \object{A} \step[2] \object{A} \\
				\step[1] \S \step[3] \xd \\
				\step[1] \O* \step[3] \O* \step[2] \O* \step[0.4] \cdot {\scriptstyle  \varphi(S(h_{(1)}),a)\varphi(S(h_{(2)}),b) } \vspace{-0.12 cm} \\
				\cd \step[2] \id \step[2] \id  \\
				\id \step[2] \x {{{}^\Psi}} \step[1.21] \id \step[0.4]  \cdot e^{-2 i \phi\delta(S(h_{(1)})^*) \delta(b^{*})}\\
				\lu[2] \step[2] \lu[2] \step[0.4] \cdot   e^{- 2 i \phi \delta(S(h_{(2)}))\delta(a)} \\
				\step[2] \Cu  \step[0.4] \cdot e^{2 i \phi \delta(S(h_{(2)}) \delta(S(h_{(1)}))} \\
				\step[4] \object{A}  \\
			\end{tangle}
\qquad.
		\end{aligned}
    \vspace{0.5 cm}
	\end{equation*}
	Both sides are equal only if:
	\begin{equation*}
		\begin{aligned}
			\varphi(\delta(h), &\delta(a) + \delta(b)) =  \varphi(\delta (S(h_{(1)})) \delta(a)) \varphi(\delta(S(h_{(2)})),\delta(b)) \cdot \\
			& \cdot e^{- 2 i \phi \delta(S(h_{(2)})) \delta(a)} e^{-2 i \phi \delta(S(h_{(1)}^*)) \delta(b^*)} e^{-2 i \phi \delta(S(h_{1}))\delta(S(h_{(2)}))} e^{2 i \phi \delta(S(h_{1}))\delta(S(h_{(2)}))} \\
		\end{aligned}
	\end{equation*}
	Using the properties of $\delta$ with respect to the antipode and the conjugation the above expression can be simplified to
	\begin{equation}
		\label{s_comp2}
		\varphi(\delta(h), \delta(a) \!+\! \delta(b)) = \varphi(\delta(h_{(1)}),\delta(a)) \varphi(\delta(h_{(2)}), \delta(b)) e^{- 2 i \phi \delta(h_{(2)}) \delta(a)} e^{-2 i \phi \delta(h_{(1)}) \delta(b)}.
	\end{equation}


\vfill
	
	\begin{thebibliography}{70}
       		\bibitem{BJR21} S.Bhattacharjee, S.Joardar, S.Roy, \emph{Braided quantum symmetries of graph $C^\ast$-algebras} arXiv:2201.09885  
       		
       		\bibitem{arek} A. Bochniak, A. Sitarz, \emph{Braided Hopf algebras from twisting},  Journal of Algebra and Its Applications 18(9),  (2018)
       		
		\bibitem{qGroups1} V.Chari, A.Pressley, \emph{A Guide to Quantum Groups}, Cambridge University Press, Cambridge (1994)

		\bibitem{dirac} L. Dabrowski, G. Landi, A. Sitarz, W. Suijlekom, J. Varilly, \emph{The Dirac operator on $SU_q(2)$}, Comm. Math. Phys., 259: 729-759 (2005)

	\bibitem{diracds} L. Dabrowski, A. Sitarz, \emph{Dirac operator on the standard Podle\'s quantum sphere},
Banach Center Publications 61 (2003), 49-58

	\bibitem{dirac_all_podles} L. Dabrowski, F. D'Andrea, G. Landi, E. Wagner, \emph{Dirac operators on all Podles quantum spheres}, Journal of Noncommutative Geometry, 1(2): 213-239 (2007)

		
		\bibitem{Hopf1} A. Klimyk, K. Schm\"{u}dgen, \emph{Quantum Groups and Their Representation}, Springer-Verlag, Berlin Heidelberg (1997) 

		\bibitem{qGroups2} S. Majid, \emph{Foundations of quantum group theory}, Cambridge University Press, Cambridge (1995)

		\bibitem{brGeo}  S. Majid, \emph{Introduction to braided geometry and $q$-Minkowski space}, in: Proceedings of the School on Quantum Groups, Varenna, 1994 

		\bibitem{su2} P. Kasprzak, R. Meyer, R. Sutanu, S. Woronowicz, \emph{Braided Quantum $SU(2)$ Groups}, Journal of Noncommutative Geometry, 10(4): 1611-1625, (2016)

		\bibitem{podles0} P. Podle\'s, \emph{Quantum spheres}, Lett. Math. Phys., 14(3): 193-202, (1987)

		\bibitem{podles} P. So\l{}tan, \emph{Podle\'s Spheres for the Braided Quantum $SU(2)$}, Linear Algebra and its Applications, 591: 169-204, (2020)

  
		
		\bibitem{quasiT} N.Reshetikhin, \emph{Quasitriangularity of Quantum Groups at Roots of 1},  Comm. Math. Phys., 170:  79-99 (1995)
		
		
	\end{thebibliography}
\end{document}